\documentclass[11pt,a4paper]{article}
\usepackage{amsthm}	
\usepackage{amsfonts, amssymb, amsmath, amscd, latexsym}
\usepackage{eucal, enumerate, latexsym}
\usepackage{graphics}
\usepackage{graphicx}
\usepackage{mathscinet}
\usepackage{psfrag}
\usepackage{tikz}
\usetikzlibrary{decorations.markings}

\newcommand{\R}{\mathbb R}
\newcommand{\N}{\mathbb N}
\newtheorem{theorem}{Theorem}[section]
\newtheorem{corollary}{Corollary}[section]
\newtheorem{proposition}{Proposition}[section]
\newtheorem{definition}{Definition}[section]
\newtheorem{remark}{Remark}[section]
\newtheorem{lemma}{Lemma}[section]

\newcommand{\ds}{\displaystyle}

\renewcommand{\epsilon}{\varepsilon}
\newcommand{\eps}{\epsilon}
\renewcommand{\phi}{\varphi}
\newenvironment{proofof}[1]{\smallskip\noindent\emph{Proof of #1.}%
\hspace{1pt}}{\hspace{-5pt}{\nobreak\quad\nobreak\hfill\nobreak%
$\square$\vspace{8pt}\par}\smallskip\goodbreak}

\makeatletter

\newlength{\captionwidth}
\long\def\@makecaption#1#2{%
   \vskip 10\p@
   \setbox\@tempboxa\hbox{#1: #2}%
   \ifdim \wd\@tempboxa > \captionwidth 
       \hbox to\hsize{\hfil
       \parbox[t]{\captionwidth}{
       \small#1: \small#2\par}
       \hfil}
     \else
       \hbox to\hsize{\hfil\box\@tempboxa\hfil}%
   \fi}

\makeatother

\usetikzlibrary{decorations.pathreplacing,decorations.markings}
\tikzset{
  on each segment/.style={
    decorate,
    decoration={
      show path construction,
      moveto code={},
      lineto code={
        \path [#1]
        (\tikzinputsegmentfirst) -- (\tikzinputsegmentlast);
      },
      curveto code={
        \path [#1] (\tikzinputsegmentfirst)
        .. controls
        (\tikzinputsegmentsupporta) and (\tikzinputsegmentsupportb)
        ..
        (\tikzinputsegmentlast);
      },
      closepath code={
        \path [#1]
        (\tikzinputsegmentfirst) -- (\tikzinputsegmentlast);
      },
    },
  },
  mid arrow/.style={postaction={decorate,decoration={
        markings,
        mark=at position .6 with {\arrow[#1]{stealth}}
      }}},
}

\begin{document}

\title{Uniqueness and nonuniqueness of fronts for degenerate diffusion-convection reaction equations}

\author{
Diego Berti\footnote{Department of Sciences and Methods for Engineering, University of Modena and Reggio Emilia, Italy}
\and
Andrea Corli\footnote{Department of Mathematics and Computer Science, University of Ferrara, Italy}
\and
Luisa Malaguti\footnotemark[1]
}



\maketitle

\begin{abstract}
We consider a scalar parabolic equation in one spatial dimension. The equation is constituted by a convective term, a reaction term with one or two equilibria, and a positive diffusivity which can however vanish. We prove the existence and several properties of traveling-wave solutions to such an equation. In particular, we provide a sharp estimate for the minimal speed of the profiles and improve previous results about the regularity of wavefronts. Moreover, we show the existence of an infinite number of semi-wavefronts with the same speed.

\vspace{1cm}

\end{abstract}

%

\section{Introduction}\label{s:I}

We study the existence and qualitative properties of traveling-wave solutions to the scalar diffusion-convection-reaction equation
\begin{equation}\label{e:E}
\rho_t + f(\rho)_x=\left(D(\rho)\rho_x\right)_x + g(\rho), \qquad t\ge 0, \, x\in \R.
\end{equation}
Here $\rho=\rho(t,x)$ is the unknown variable and takes values in the interval $[0,1]$. The convective term $f$ satisfies the condition
\begin{itemize}
\item[{(f)}]\, $f\in C^1[0,1]$, $f(0)=0$.
\end{itemize}
\noindent
The requirement $f(0)=0$ is not a real assumption, since $f$ is defined up to an additive constant; we denote $h(\rho)=\dot f(\rho)$, where with a dot we intend the derivative with respect to the variable $\rho$ (or $\phi$ later on). About the diffusivity $D$ and the reaction term $g$ we consider two different scenarios, where the assumptions are made on the pair $D$, $g$; we assume either
\begin{itemize}
\item[{(D1)}] \, $D\in C^1[0,1]$, $D>0$ in $(0,1)$ and $D(1)=0$,

\item[{(g0)}]\, $g\in C^0[0,1]$, $g>0$ in $(0,1]$, $g(0)=0$,

\end{itemize}

or else

\begin{itemize}
\item[{(D0)}] \, $D\in C^1[0,1]$, $D>0$ in $(0, 1)$ and $D(0)=0$,

\item[{(g01)}]\, $g\in C^0[0,1]$, $g>0$ in $(0,1)$, $g(0)=g(1)=0$.

\end{itemize}
In the above notation, the numbers suggest where it is {\em mandatory} that the corresponding function vanishes. Notice that (D1) leaves open the possibility for $D$ to vanish or not at $0$, and (D0) for $D$ at $1$. We refer to Figure \ref{f:f} for a graphical illustration of these assumptions. Notice that the product $Dg$ always vanishes at both $0$ and $1$ under both set of assumptions.

\smallskip

\begin{figure}[htb]
\begin{center}

\begin{tikzpicture}[>=stealth, scale=0.6]
\draw[->] (0,0) --  (6,0) node[below]{$\rho$} coordinate (x axis);
\draw[->] (0,0) -- (0,4) node[right]{$f$} coordinate (y axis);
\draw[thick] (0,0) .. controls (1.5,4) and (3.5,4) .. (5,1) ;
\draw[dotted] (5,0) -- (5,1);
\draw(5,0)  node[below]{\footnotesize{$1$}};

\begin{scope}[xshift=7cm]
\draw[->] (0,0) --  (6,0) node[below]{$\rho$} coordinate (x axis);
\draw[->] (0,0) -- (0,4) node[right]{$D$} coordinate (y axis);
\draw[thick] (0,2) .. node[near start,above]{\footnotesize{(D1)}} controls (1,3) and (4,3) .. (5,0); 
\draw[thick, dashed] (0,0) .. node[very near end,above]{\footnotesize{(D0)}} controls (1,2) and (2,2.5) .. (5,3);
\draw(5,0) node[below]{$1$};
\draw[dotted] (5,0) -- (5,3);
\end{scope}

\begin{scope}[xshift=14cm]
\draw[->] (0,0) --  (6,0) node[below]{$\rho$} coordinate (x axis);
\draw[->] (0,0) -- (0,4) node[right]{$g$} coordinate (y axis);
\draw[thick] (0,0) .. node[near end,above]{\footnotesize{(g0)}} controls (0.2,2.5) and (3,2.8) .. (5,3); 
\draw[thick,dashed] (0,0) .. node[midway,below]{\footnotesize{(g01)}} controls (1,2.5) and (4,2.5) .. (5,0); 
\draw(5,0) node[below]{\footnotesize{$1$}};
\draw[dotted] (5,0) -- (5,3);
\end{scope}

\end{tikzpicture}

\end{center}
\caption{\label{f:f}{Typical plots of the functions $f$, $D$ and $g$. In the plots of $D$ and $g$, solid or dashed lines depict pairs of functions $D$ and $g$ that are considered together in the following. The possibility that $D$ vanishes at the other extremum is left open.}}
\end{figure}
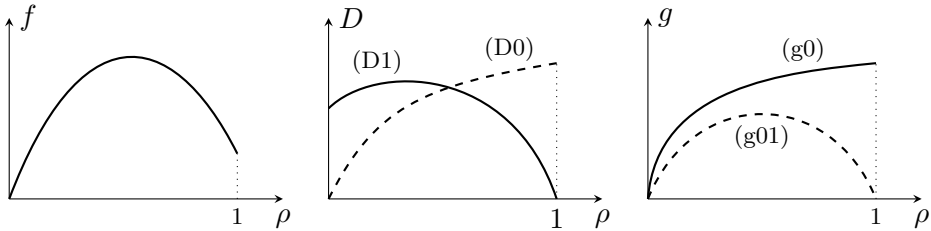

We also require the following condition on the product of $D$ and $g$:
\begin{equation}
\label{Dg}
\limsup_{\varphi \to 0^+}\frac{D(\varphi)g(\varphi)}{\varphi} < +\infty,
\end{equation}
which is equivalent to $D(\varphi)g(\varphi)\le L\phi$, for some $L>0$ and $\phi$ in a right neighborhood of $0$.

\smallskip

In \eqref{e:E}, the notation $\rho=\rho(t,x)$ suggests a density; this is indeed the case. Recently, the modeling of collective movements has attracted the interest of several mathematicians \cite{Garavello-Han-Piccoli_book, Garavello-Piccoli_book, Rosinibook}. This paper is partly motivated by such a research stream and carries on the analysis of a scalar parabolic model begun in \cite{CdRMR, CM-DPDE, CM-ZAMP}. Indeed, if $f(\rho)=\rho v(\rho)$, where the velocity $v$ is an assigned function, then equation \eqref{e:E} can be understood as a simplified model for a crowd walking with velocity $v$ along a straight path with  side entries for other pedestrians, which are modeled by $g$; here $\rho$ is understood as the crowd normalized density. Assumption (g01), for instance, means that pedestrians do not enter if the road is empty ($g(0)=0$, modeling an aggregative behavior) or if it is fully occupied ($g(1)=0$, because of lack of space). If the diffusivity is small, then the diffusion term accounts for some \lq\lq chaotic\rq\rq\ behavior, which is common in crowds movements. In this framework, $D$ may degenerate at the extrema of the interval where it is defined \cite{Bellomo-Delitala-Coscia, BTTV, Nelson_2000}; for more details we refer to \cite{CM-DPDE}.
The assumption (g0) is better motivated by population dynamics. In this case $g$ is a growth term which, for instance, increases with the population density $\rho$. We refer to \cite{Murray} for analogous modelings in biology. Anyhow, apart from the above possible applications, equation \eqref{e:E} is a quite general diffusion-convection-reaction equation that deserves to be fully understood.

\smallskip

A {\em traveling-wave solution} is, roughly speaking, a solution to \eqref{e:E} of the form $\rho(t,x)= \phi(x-ct)$, for some profile $\phi=\phi(\xi)$ and constant wave speed $c$, see \cite{GK} for general information. In this case the profile must satisfy, in some sense, the equation
\begin{equation}\label{e:ODE}
\left(D(\phi)\phi^{\prime}\right)^{\prime}+\left(c -h(\phi)\right)\phi' + g(\phi)=0,
\end{equation}
where ${'}$ denotes the derivative with respect to $\xi$.
We consider in this paper non-constant, monotone profiles, and focus on the case they are decreasing. As a consequence, we aim at determining solutions to \eqref{e:ODE} whose values at $\pm\infty$ are the zeroes of the function $g$ and then satisfy either
\begin{equation}\label{e:infty}
\phi(-\infty)=1, \qquad \phi(+\infty)=0,
\end{equation}
or simply
\begin{equation}\label{e:infty0}
\phi(+\infty)=0,
\end{equation}
according to we make assumption (g01) or (g0). The former profiles are called {\em wavefronts}, the latter are {\em semi-wavefronts}; precise definitions are provided in Definition \ref{d:tws}. Notice that in both cases the equilibria may be reached for a finite value of the variable $\xi$ as a consequence of the degeneracy of $D$ at those points. These solutions represent single-shape smooth transitions between the two constant densities $0$ and $1$. The interest of wavefronts lies in the fact that they are viscous approximations of shock waves to the inviscid version of equation \eqref{e:E}, i.e., when $D=0$. Semi-wavefronts lack of this motivation but are nevertheless meaningful for applications \cite{CM-DPDE}; moreover, wavefronts connecting \lq\lq nonstandard\rq\rq\ end states can be constructed by pasting semi-wavefronts \cite{CM-ZAMP}. At last, we point out that assumption \eqref{Dg} is usual in this framework, when looking for decreasing profiles, see e.g. \cite{Aronson-Weinberger}.

%
%
%
%
%

\smallskip

If $D(\rho)\ge0$, the existence of solutions to the initial-value problem for \eqref{e:E} is more or less classical \cite{Vazquez}; however, the {\em fine structure} of traveling waves reveals a variety of different patterns. We refer to \cite{Malaguti-Marcelli_2002, MMconv}, respectively, for the cases where $D$ is non degenerate, i.e., $D>0$, and for the degenerate case, where $D$ can vanish at either $0$ or $1$. The main results of those papers is that there is a critical threshold $c^*$, depending on both $f$ and the product $Dg$, such that traveling waves satisfying \eqref{e:infty} exist if and only if $c\ge c^*$. The smoothness of the profiles depend on $f$, $D$ and $c$ but not on $g$. In both papers the source term satisfies (g01); see \cite{CdRMR, CM-DPDE} for the case when $g$ has only one zero.

The case when $D$ changes sign, which is not studied in this paper, also has strong motivations: we quote \cite{Maini-Malaguti-Marcelli-Matucci2006, Padron} for biological models and \cite{CM-ZAMP} for applications to collective movements.
Several results about traveling waves have been obtained in \cite{CM-ZAMP, Ferracuti-Marcelli-Papalini, Kuzmin-Ruggerini, Maini-Malaguti-Marcelli-Matucci2006, Maini-Malaguti-Marcelli-Matucci2007}.

\smallskip

In this paper we study semi-wavefronts and wavefronts for \eqref{e:E}, thus completing the analysis of \cite{CdRMR, CM-DPDE}. We prove that in both cases there is a threshold $c^*$ such that profiles only exists for $c\ge c^*$; we also study their regularity and strict monotonicity, namely whether they are {\em classical} (i.e., $C^1$) or {\em sharp} (and then reach an equilibrium at a finite $\xi$ in a no more than continuous way). We strongly rely on \cite{ Malaguti-Marcelli_2002, MMconv} and exploit some recent results obtained in \cite{Marcelli-Papalini}. Several examples are scattered throughout the paper to show that our assumptions are necessary in most cases.

This research has some important novelties. First, we give a refined estimate for $c^*$, which allows to better understand the meaning of this threshold. Second, we improve a result obtained in \cite{MMconv} about the appearance of wavefronts with a sharp profile. Third, in the case of semi-wavefronts, we show that for any speed $c\ge c^*$ there exists a family of profiles with speed $c$. This phenomenon does not show up in \cite{CdRMR, CM-DPDE}.

The main tool to investigate \eqref{e:ODE} is the analysis of singular first-order problems as
\begin{equation}
\label{first order problem000}
\left\{
\begin{array}{ll}
\dot{z}(\varphi)=h(\varphi)-c-\frac{D(\varphi) g(\varphi)}{z(\varphi)}, &\varphi\in (0,1),\\
z(\varphi) < 0, & \varphi \in (0,1),
\\
z(0)=0.
\end{array}
\right.
\end{equation}
Problem \eqref{first order problem000} is deduced by problem \eqref{e:ODE}-\eqref{e:infty0} by the singular change of variables $z(\phi) := D(\phi)\phi'$, where the right-hand side is understood to be computed at $\phi^{-1}(\phi)$, see e.g. \cite{CM-DPDE, Malaguti-Marcelli_2002}. Notice that $\phi^{-1}$ exists by the assumption of monotony of $\phi$.

On the other hand, the analysis of problem \eqref{first order problem000} is fully exploited in the forthcoming paper \cite{BCM2}, which deals with the case in which $D$ changes sign once. In that paper we show that there still exist wavefronts joining $1$ with $0$, which travel across the region where $D$ is negative; they are constructed by pasting two semi-wavefronts obtained in the current paper. Similar results in the case $g=0$ are proved in \cite{CM-ZAMP}.

\smallskip

Here is an account of the paper. In Section \ref{sec:main} we provide some basic definitions and state our main results. The analysis of problem \eqref{first order problem000} and of other related singular problems occupies Sections \ref{sec:first order 1} to \ref{ssec:dotz at zero}. Then, in Sections \ref{s:existence_0alpha} and \ref{s:rnew egularity of fronts} we exploit such results to construct semi-wavefronts and wavefronts, respectively; there, we prove our main results.

\section{Main results}\label{sec:main}
\setcounter{equation}{0}

We give some definitions on traveling waves and their profiles. Let $I\subseteq \R$ be an open interval.

\begin{definition}\label{d:tws} Assume $f,D,g\in C[0,1]$. Consider a function $\varphi\in C(I)$ with values in $[0,1]$, which is differentiable a.e. and such that $D(\varphi) \varphi^{\, \prime}\in L_{\rm loc}^1(I)$; let $c$ be a real constant. The function $\rho(x,t):=\varphi(x-ct)$, for $(x,t)$ with $x-ct \in I$, is a {\em traveling-wave} solution of equation \eqref{e:E} with wave speed $c$ and wave profile $\phi$ if, for every $\psi\in C_0^\infty(I)$,
\begin{equation}\label{e:def-tw}
\int_I \left(D\left(\phi(\xi)\right)\phi'(\xi) - f\left(\phi(\xi)\right) + c\phi(\xi) \right)\psi'(\xi) - g\left(\phi(\xi)\right)\psi(\xi)\,d\xi =0.
\end{equation}
\end{definition}
\noindent
Definition \ref{d:tws} can be made more precise. Below, {\em monotonic} means that $\phi(\xi_1)\le \phi(\xi_2)$ (or $\phi(\xi_1)\ge \phi(\xi_2)$) for every $\xi_1<\xi_2$ in the domain of $\phi$; in {\em (iii)} we assume $g(0)=g(1)=0$, while in {\em (iv)} we only require that $g$ vanishes at the point which is specified by the semi-wavefront. A traveling-wave solution is

\begin{enumerate}[{\em (i)}]

\item {\em global} if $I=\R$ and {\em strict} if $I\ne \R$ and $\phi$ is not extendible to $\R$;

\item {\em classical} if $\varphi$ is differentiable, $D(\varphi) \varphi'$ is absolutely continuous and \eqref{e:ODE} holds a.e.; {\em sharp at $\ell$} if there exists $\xi_{\ell}\in I$ such that $\phi(\xi_{\ell})=\ell$, with $\phi$ classical in $I\setminus\{\xi_\ell\}$ and not differentiable at $\xi_{\ell}$;

\item {\em a wavefront} if it is global, with a monotonic, non-constant profile $\phi$ satisfying either \eqref{e:infty} or the converse condition;

\item {\em a semi-wavefront to} $1$ (or {\em to} $0$) if $I=(a,\infty)$ for $a\in\R$, the profile $\phi$ is monotonic, non-constant and $\phi(\xi)\to1$ (respectively, $\phi(\xi)\to0$) as $\xi\to\infty$; {\em a semi-wavefront from} $1$ (or {\em from} $0$) if $I=(-\infty,b)$ for $b\in\R$, the profile $\phi$ is monotonic, non-constant and $\phi(\xi)\to1$ (respectively, $\phi(\xi)\to0$) as $\xi\to-\infty$.

\end{enumerate}

\noindent In {\em (iv)} we say that $\phi$ connects $\phi(a^+)$ ($1$ or $0$) with $1$ or $0$ (resp., with $\phi(b^-)$).

The smoothness of a profile depends on the degeneracy of $D$, see \cite{GK}. More precisely, assume (f), and either (D1), (g0) or (D0), (g01); let $\rho$ be any traveling-wave solution of \eqref{e:E} with profile $\phi$ defined in $I$ and speed $c$. Then $\phi$ is classical in each interval $J \subset I$ where $D\left(\phi(\xi)\right)> 0$ for $\xi \in J$, and $\phi\in C^2\left(J\right)$. Profiles are determined up to a space shift.

Our first main result concerns {\em semi-wavefronts}.

\begin{theorem}
\label{th:swf to zero}
Assume {\em (f)}, {\em (D1)}, {\em (g0)} and \eqref{Dg}. Then, there exists $c^*\in\R$, which satisfies
\begin{equation}
\label{estimates on c* new}
\max\left\{\sup_{\phi \in (0,1]} \frac{f(\phi)}{\phi}, h(0) + 2\sqrt{\liminf_{\phi \to 0^+} \frac{D(\phi)g(\phi)}{\phi}} \right\} \leq c^* \leq 2\sqrt{\sup_{\varphi \in (0,1]} \frac{D(\phi)g(\phi)}{\varphi}}+ \sup_{\varphi \in (0,1]} \frac{f(\phi)}{\phi},
\end{equation}
such that \eqref{e:E} has strict semi-wavefronts to $0$, connecting $1$ to $0$, if and only if $c\geq c^*$.
\par
Moreover, if $\phi$ is the profile of one of such semi-wavefronts, then it holds that
\begin{equation}
\label{strict monotonicity}
\varphi\rq{}(\xi) < 0 \ \mbox{ for any } \ 0<\varphi(\xi) < 1.
\end{equation}
\end{theorem}

For a fixed $c>c^*$, the profiles of Theorem \ref{th:swf to zero} are {\em not} unique. This lack of uniqueness is not due only to the action of space shifts but, more intimately, to the non-uniqueness of solutions to problem \eqref{first order problem000} that is proved in Proposition \ref{prop: first order ii} below. Roughly speaking, these profiles depend on a parameter $b$ ranging in the interval $[\beta(c),0]$, for a suitable threshold $\beta(c)\le0$. As a conclusion, the family of profiles can be precisely written as
\begin{equation}\label{e:beta1}
\phi_b=\phi_b(\xi), \quad \hbox{ for }b\in[\beta(c),0].
\end{equation}
Moreover, $\beta(c)<0$ if $c>c^*$ and $\beta(c)\to-\infty$ as $c\to+\infty$. The threshold $\beta(c)$ essentially corresponds to the minimum value that the quantity $D(\phi_b)\phi_b'$ achieves when $\phi_b$ reaches $1$, for $b\in[\beta(c),0]$. This loss of uniqueness is a novelty if we compare Theorem \ref{th:swf to zero} with analogous results in \cite{CdRMR, CM-DPDE}. In particular, in \cite[Theorem 2.7]{CM-DPDE} the assumptions on the functions $D$ and $g$ are {\em reversed}: both of them are positive in $(0,1)$ with $D(0)=0<g(0)$, $D(1)>0=g(1)$; in \cite[Theorem 2.3]{CdRMR} $D$ and $g$ are still positive in $(0,1)$ but the vanishing conditions are  $D(1)=0=g(1)$. In both cases the profiles exist for every $c\in\R$ and are unique. The different results are due to the nature of the equilibria of the dynamical systems of \eqref{e:ODE}.

\smallskip

The estimates \eqref{estimates on c* new} deserve some comments. The left estimate improves analogous bounds (see \cite{Marcelli-Papalini} for a comprehensive list) by including the term $\sup_{\phi \in (0,1]}f(\phi)/\phi\ge h(0)$ on the left-hand side. This improvement looks more significative if we also assume $\dot{(Dg)}(0)=0$, as we do in the Theorem \ref{th:regularity wf}. In this case \eqref{estimates on c* new} reduces to
\begin{equation}
\label{e:great0}
\sup_{\phi \in (0,1]} \frac{f(\phi)}{\phi} \le c^* \leq 2\sqrt{\sup_{\varphi \in (0,1]} \frac{D(\phi)g(\phi)}{\varphi}}+ \sup_{\varphi \in (0,1]} \frac{f(\phi)}{\phi}.
\end{equation}
which can be written with obvious notation as
\[
c_{con}\le c^*\le c_{dr} + c_{con},
\]
where the indexes label velocities related to the convection or diffusion-reaction components. In \eqref{e:great0} the {\em same} term, accounting for the dependence on $f$, occurs in {\em both} the lower and upper bound. This symmetry, which shows the shift of the critical threshold as a consequence of the convective term $f$, occurs in {\em none} of the previous estimates.

The meaning of $c_{dr}$ is known since \cite{Aronson-Weinberger}; we comment on $c_{con}$. In the diffusion-convection case (i.e., when $g=0$), there exist profiles connecting $\ell\in(0,1]$ to $0$ if and only if
\begin{equation}\label{e:line_condition}
s_\ell(\phi):= \frac{f(\ell)}{\ell}\phi>f(\phi),\quad \hbox{ for } \phi\in(0,\ell),
\end{equation}
see \cite[Theorem 9.1]{GK}. The quantity $c_{con}$ then represents the {\em maximal speed} that can be reached by the profiles connecting $\ell$ to $0$, for $\ell\in(0,1]$. Condition \eqref{e:line_condition} is also necessary and sufficient in the purely hyperbolic case (i.e., when also $D=0$) in order that the equation $u_t+f(u)_x=0$ admits a shock wave of speed $f(\ell)/\ell$ with $\ell$ as left state and $0$ as right state. This is not surprising since the viscous profiles approximate the shock wave and converge to it in the vanishing viscosity limit. Indeed, condition \eqref{e:line_condition} does not depend on $D$.

The presence of the positive reaction term $g$ satisfying (g01) (if (g0) holds we only have semi-wavefronts, but the same bounds still hold) does not allow profile speeds to be less than $c_{con}$: assuming that $z$ satisfies $\eqref{first order problem000}$, by the positivity of both $D$ and $g$ we deduce
\begin{equation}\label{e:zaz}
c\ge \sup_{\phi\in(0,1]}\left(\frac{f(\phi)}{\phi} - \frac{z(\phi)}{\phi}\right)\ge c_{con}.
\end{equation}
Then, $c_{con}$ now becomes a bound for the {\em minimal speed} of the profiles. The bound \eqref{e:zaz} is strict (i.e., there is a gap between $c_{con}$ and $c^*$) if $\dot{\left(Dg\right)}(0)>0$; this occurs for instance if $D(0)>0$ and $\dot{g}(0)>0$ and follows by integrating $\eqref{first order problem000}_1$ from $0$ to $\phi$ and \eqref{estimates on c* new}, see Remark \ref{r:gap}. If $f=0$, then the corresponding strict bound $c^*>0$ occurs for any positive and continuous $D$ and $g$: if $c^*=0$ then $z$ should be an increasing function by \eqref{first order problem}, a contradiction.

In some cases, semi-wavefronts are sharp at $0$. We refer to Corollary \ref{cor:swf qualitative behavior} for a detailed account of the behavior of the profiles when they reach the equilibrium.

\smallskip

We now present our result on {\em wavefronts}; we assume that $D$ and $g$ satisfiy (D0) and (g01). The goal is to extend results contained in \cite[Theorems 2.1 and 6.1]{MMconv} regarding the existence and, more importantly, the regularity of wavefronts of Equation \eqref{e:E}. In particular, the next theorem has the merit to derive the classification of wavefronts under (D0), merely, without additional assumptions (which were instead required in \cite[Theorems 2.1 and 6.1]{MMconv}). Notice that in the following result we require that $D$ vanishes at $0$; this assumption leads to improve not only the left-hand bound \eqref{estimates on c* new} on $c^*$ by \eqref{e:great0}, but also the right-hand bound, by means of a recent integral estimate provided in \cite{Marcelli-Papalini}.

\begin{theorem}
\label{th:regularity wf}
Assume {\em (f)}, {\em (D0)} and {\em (g01)} and \eqref{Dg}. Then there exists $c^*$, satisfying
\begin{equation}
\label{e:est c*3}
\sup_{\phi \in (0,1]} \frac{f(\phi)}{\phi} \le c^* \le \sup_{\phi \in (0,1]} \frac{f(\phi)}{\phi} + 2\sqrt{\sup_{\phi\in (0,1]} \frac1{\phi}\int_{0}^{\phi} \frac{D(\sigma)g(\sigma)}{\sigma}\,d\sigma},
\end{equation}
such that Equation \eqref{e:E} admits a (unique up to space shifts) wavefront, whose wave profile $\phi$ satisfies \eqref{e:infty}, if and only if $c\ge c^*$.
Moreover, we have $\phi\rq{}(\xi) <0$, for $0<\phi(\xi) <1$, and
\begin{enumerate}[(i)]
\item if $c>c^*$, then $\phi$ is classical at $0$;

\item if $c=c^*$ and $c^*>h(0)$, then $\phi$ is sharp at $0$ and if it reaches $0$ at $\xi_0 \in \R$ then
\begin{equation*}
\lim_{\xi \to \xi_0^-} \phi\rq{}(\xi)=
\left\{
\begin{array}{ll}
\frac{h(0)-c^*}{\dot{D}(0)} <0 \ &\mbox{ if } \ \dot{D}(0)>0,\\[2mm]
-\infty \ &\mbox{ if } \ \dot{D}(0)=0.
\end{array}
\right.
\end{equation*}
\end{enumerate}
\end{theorem}

As in analogous cases \cite{CM-DPDE}, Theorem \ref{th:regularity wf} provides no information about the smoothness of the profiles when $c=c^*=h(0)$. We show in Remark \ref{r:final?} that in such a case profiles may be either sharp or classical.

\section{Singular first-order problems}
\label{sec:first order 1}
\setcounter{equation}{0}

Here we begin the analysis of problem \eqref{first order problem000}.
First, we consider, for $c\in\R$, the problem
\begin{equation}
\label{first order problem0}
\begin{cases}
\dot{z}(\varphi)=h(\varphi)-c-\frac{q(\varphi)}{z(\varphi)}, \ \varphi\in (0,1),\\
z(\varphi) < 0 , \ \varphi \in (0,1),
\end{cases}
\end{equation}
where we assume
\begin{equation}\label{e:q}
q \in C^0[0,1]\quad \hbox{ and }\quad q>0 \hbox{ in } (0,1).
\end{equation}
We point out that the differential equation $\eqref{first order problem0}_1$ generalizes $\eqref{first order problem000}_1$ since the assumptions on $q$ are a bit less strict than the ones on $Dg$, under (D1)-(g0) or (D1)-(g01). In the following lemma we prove that a solution of \eqref{first order problem0} can be extended continuously up to the boundary.

\begin{lemma}
\label{lem:zlimit}
Assume \eqref{e:q}. If $z\in  C^1(0,1)$ is a solution of \eqref{first order problem0}, then it can be extended continuously to the interval $[0,1]$.
\end{lemma}

\begin{proof}
Since $q/z < 0$ in $(0,1)$, then for any $0<\phi < \phi_1 < 1$ the function
\begin{equation*}
\phi \to
\int_{\phi}^{\phi_1} \frac{q(\sigma)}{z(\sigma)}\,d\sigma
\end{equation*}
is strictly increasing. Hence, we can pass to the limit as $\phi \to 0^+$ in the expression
\begin{equation}
\label{e:zlimit}
z(\phi) = z(\phi_1) - \int_{\phi}^{\phi_1} \left(h(\sigma)-c\right)\,d\sigma + \int_{\phi}^{\phi_1} \frac{q(\sigma)}{z(\sigma)}\,d\sigma,
\end{equation}
which is obtained by integrating $\eqref{first order problem0}_1$ in $(\phi,\phi_1)$. Then $z(0^+)$ exists and necessarily lies in $[-\infty, 0]$ because of $\eqref{first order problem0}_2$. If $z(0^+)=-\infty$, then by passing to the limit for $\phi \to 0^+$ in \eqref{e:zlimit} we find a contradiction, since the last integral converges as $\phi \to 0^+$. Hence, $z(0^+)\in (-\infty, 0]$.
\par
For $z(1^-)$ the proof is even simpler: by integrating $\eqref{first order problem0}_1$ in $(\phi_2,\phi)$, for $0<\phi_2<\phi<1$, we obtain \eqref{e:zlimit} with $\phi_2$ replacing $\phi_1$. As before, we deduce that $z(1^-)$ exists. Also, since the last integral in \eqref{e:zlimit} is now positive, we get $z(\phi) > z(\phi_2) + \int_{\phi_2}^{\phi} \left(h(\sigma) - c\right) \,d\sigma$, for any $\phi\in(\phi_2,1)$. This directly rules out the alternative $z(1^-)=-\infty$ and concludes the proof.
\end{proof}

We now summarize \cite[Lemmas 4.1 and 4.3]{CM-DPDE} in a version for our purposes, by also exploiting Lemma \ref{lem:zlimit}. These tools were obtained in \cite{CM-DPDE} under stricter assumptions on $q$, but it is easy to verify that they also apply to the current case, in virtue of \eqref{e:q}. For $\mu<0$ and $\sigma\in(0,1]$ or $\sigma\in[0,1)$, they deal with the systems
\begin{equation}
	\label{e:sol from sigma}
	\begin{cases}
	\dot{z}(\phi)=h(\phi)-c-\frac{q(\phi)}{z(\phi)}, \ \phi < \sigma,\\
	z(\sigma)=\mu,
	\end{cases}
	\quad
	\begin{cases}
	\dot{z}(\phi)=h(\phi)-c-\frac{q(\phi)}{z(\phi)}, \ \phi > \sigma,\\
	z(\sigma)=\mu.
	\end{cases}
	\end{equation}
A function $\eta\in C^1(\sigma_1,\sigma_2)$, for $0\le \sigma_1 < \sigma_2 \le 1$, is an upper-solution of $\eqref{first order problem0}_1$ in $(\sigma_1,\sigma_2)$ if
\begin{equation}
\label{e:upper-sol}
\dot{\eta}(\phi) \ge  h(\phi) - c -\frac{q(\phi)}{\eta(\phi)} \ \mbox{ for any }  \ \sigma_1 < \phi < \sigma_2.
\end{equation}
The upper-solution $\eta$ is said strict if the inequality in \eqref{e:upper-sol} is strict. A function $\omega \in C^1(\sigma_1,\sigma_2)$ is a (strict) lower-solution of $\eqref{first order problem0}_1$ in $(\sigma_1,\sigma_2)$ if the (strict) inequality in \eqref{e:upper-sol} is reversed.

\begin{lemma}
\label{lem:cm-dpde}
Assume \eqref{e:q} and consider equation $\eqref{first order problem0}_1$; the following results hold.
\begin{enumerate}
\item
Set $\mu<0$. Then,
	{
	\begin{enumerate}
	\item let $\sigma\in (0,1]$; then problem $\eqref{e:sol from sigma}_1$
		admits a unique solution $z\in C^0[0,\sigma]\cap C^1(0,\sigma)$;
	
	 \item let $\sigma \in [0,1)$; then problem $\eqref{e:sol from sigma}_2$
		admits a unique solution $z\in C^0[\sigma,\delta]\cap C^1(\sigma,\delta)$, for some maximal $\sigma < \delta \le 1$. Moreover, either $\delta=1$ or $z(\delta)=0$.
	\end{enumerate}
	}
\item Set $0\le \sigma_1 < \sigma_2 \le 1$; let $z$ be a solution of \eqref{first order problem0} in $(\sigma_1,\sigma_2)$. It holds that:
	{
	\begin{enumerate}
	\item
	if $\eta$ is a strict upper-solution of $\eqref{first order problem0}_1$ in $(\sigma_1,\sigma_2)$, then
	\begin{enumerate}[(i)]
	
	\item if $\eta(\sigma_2) \le z(\sigma_2)<0$, then $\eta < z$ in $(\sigma_1,\sigma_2)$;
	
	\item if $0>\eta (\sigma_1) \ge z(\sigma_1)$ then $\eta > z$ in $(\sigma_1,\sigma_2)$; moreover, if $\eta$ is defined in $[0,1]$, then $z$ must be defined in $[\sigma_1,1]$ and $\eta > z$ in $(\sigma_1,1)$;
	\end{enumerate}
	\item
	if $\omega$ is a strict lower-solution of $\eqref{first order problem0}_1$ in $(\sigma_1,\sigma_2)$, then
	\begin{enumerate}[(i)]
	
	\item if $0>\omega(\sigma_2) \ge z(\sigma_2)$, then $\omega > z$ in $(\sigma_1,\sigma_2)$; moreover, if $\omega$ is defined in $[0,1]$, then $z$ must be defined in $[0,\sigma_2]$ and $\omega > z$ in $(0,\sigma_2)$;
	
	\item if $\omega (\sigma_1) \le z(\sigma_1)<0$ then $\omega < z$ in $(\sigma_1,\sigma_2)$.
	\end{enumerate}
	\end{enumerate}
	}
\end{enumerate}
\end{lemma}

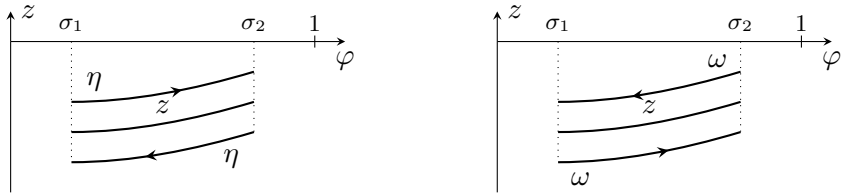
\begin{figure}[htb]
\begin{center}
\begin{tikzpicture}[>=stealth, scale=0.8]

\draw[->] (0,0) --  (5.5,0) node[below]{$\phi$} coordinate (x axis);
\draw[->] (0,0) -- (0,0.5) node[right]{$z$} coordinate (y axis);
\draw (0,0) -- (0,-2.5);
\draw (5,0.1) -- (5,-0.1);
\draw (5,0)  node[above]{\footnotesize{$1$}};
\draw[dotted] (1,0) node[above]{\footnotesize{$\sigma_1$}}-- (1,-2);
\draw[dotted] (4,0) node[above]{\footnotesize{$\sigma_2$}}-- (4,-1.5);

\draw[thick] (1,-1.5) .. controls (2,-1.5) and (3,-1.3) .. node[midway, above]{$z$} (4,-1) ; 

\path[draw=black,thick,postaction={on each segment={mid arrow=black,thick}}]
(1,-1) .. controls (2,-1) and (3,-0.8) .. node[very near start, above]{$\eta$} (4,-0.5); 

\path[draw=black,thick,postaction={on each segment={mid arrow=black,thick}}]
(4,-1.5) .. controls (3,-1.8) and (2,-2) .. node[very near start, below]{$\eta$} (1,-2)  ; 

\begin{scope}[shift ={(8,0)}]
\draw[->] (0,0) --  (5.5,0) node[below]{$\phi$} coordinate (x axis);
\draw[->] (0,0) -- (0,0.5) node[right]{$z$} coordinate (y axis);
\draw (0,0) -- (0,-2.5);
\draw (5,0.1) -- (5,-0.1);
\draw (5,0)  node[above]{\footnotesize{$1$}};
\draw[dotted] (1,0) node[above]{\footnotesize{$\sigma_1$}}-- (1,-2);
\draw[dotted] (4,0) node[above]{\footnotesize{$\sigma_2$}}-- (4,-1.5);

\draw[thick] (1,-1.5) .. controls (2,-1.5) and (3,-1.3) .. node[midway, above]{$z$} (4,-1) ; 

\path[draw=black,thick,postaction={on each segment={mid arrow=black,thick}}]
(4,-0.5) .. controls (3,-0.8) and (2,-1) .. node[very near start, above]{$\omega$} (1,-1); 

\path[draw=black,thick,postaction={on each segment={mid arrow=black,thick}}]
(1,-2) .. controls (2,-2) and (3,-1.8) .. node[very near start, below]{$\omega$} (4,-1.5)  ; 
\end{scope}
\end{tikzpicture}
\end{center}

\caption{\label{f:sup-sub}{An illustration of Lemma \ref{lem:cm-dpde} {\em (2)}. Left: supersolutions $\eta$; right: subsolutions $\omega$.}}
\end{figure}

In the context of equations as $\eqref{first order problem0}_1$, proper limit arguments are often needed.

\begin{lemma}
\label{lem:convergence}
Assume \eqref{e:q}. Let $\{c_n\}_n$ be a sequence of real numbers and $c\in\R$ such that $c_n\to c$ as $n \to \infty$.
Let $z_n\in C^0[0,1]\cap C^1(0,1)$ satisfy \eqref{first order problem0} corresponding to $c_n$. If $\{z_n\}_n$ is increasing and there exists $v\in C^0[0,1]$ such that
\begin{equation}
\label{e:bound}
 z_n(\phi) \le v(\phi) < 0\ \mbox{ for any } \ n\in \N \ \mbox{ and } \  \phi \in (0,1),
\end{equation}
then $z_n$ converges (uniformly on $[0,1]$) to a solution $\bar{z} \in C^0[0,1]\cap C^1(0,1)$ of \eqref{first order problem0}.
\par
The same conclusion holds if $\{z_n\}_n$ is decreasing and there exists $w\in C^0[0,1]$ such that
$$
z_n(\phi)\ge w(\phi) \ \mbox{ for any } \ n\in\N \ \mbox{ and } \ \phi \in (0,1).
$$
\end{lemma}

\begin{proof}
Take first $\{z_n\}_n$ increasing. From  \eqref{e:bound}, we can define $\bar{z}=\bar{z}(\phi)$ as
$$
\lim_{n\to \infty} z_n(\phi) =: \bar{z}(\phi), \ \phi \in(0,1).
$$
It is obvious that $z_1\leq \bar{z} \leq v<0$ in $(0,1)$.
By integrating $\eqref{first order problem0}_1$, we have
$$
z_n(\phi)-z_n(\phi_0) = \int_{\phi_0}^{\phi} \left\{h(\sigma)-c_n + \frac{q(\sigma)}{-z_n(\sigma)}\right\}\,d\sigma \ \mbox{ for any } \ \phi_0,\phi \in (0,1).
$$
Since, for every $\sigma \in (0,1)$, the sequence $\left\{q(\sigma)/(-z_n(\sigma))\right\}_n$
is increasing, then the Monotone Convergence Theorem implies that
$$
\bar{z}(\phi)-\bar{z}(\phi_0) = \int_{\phi_0}^{\phi} \left\{h(\sigma)-c -\frac{q(\sigma)}{\bar{z}(\sigma)}\right\}\,d\sigma  \ \mbox{ for any } \ \phi_0,\phi \in (0,1),
$$
where all the involved quantities are finite. This tells us that $\bar{z}$ is absolutely continuous in every compact interval $[a,b]\subset (0,1)$. By differentiating, we then obtain that $\bar{z}\in C^1(0,1)$ satisfies \eqref{first order problem0}. From Lemma \ref{lem:zlimit}, we also have that $\bar{z}\in C^0[0,1]$. To conclude that $z_n$ converges to $\bar{z}$ uniformly on $[0,1]$, it only remains to prove that
\begin{equation}
\label{e:estr conv2}
\bar{z}(0^+)=\lim_{n\to \infty}z_n(0) \ \mbox{ and } \ \bar{z}(1^-)=\lim_{n\to \infty}z_n(1).
\end{equation}
Indeed, if \eqref{e:estr conv2} holds, then $\{z_n\}_n$ turns out to be a monotone sequence of continuous functions converging pointwise to $\bar{z} \in C^0[0,1]$ on a compact set. Then, by {\em Dini\rq{}s monotone convergence theorem} (see \cite[Theorem 7.13]{Rudinbook}), $z_n$ must converge uniformly to $\bar{z}$ on $[0,1]$. We prove only $\eqref{e:estr conv2}_1$ since $\eqref{e:estr conv2}_2$ follows as well. If $z_n (0)\to 0$, as $n\to \infty$, then $\bar{z}(0^+)=0$, because $z_n \le \bar{z} <0$ in $(0,1)$. Hence $\eqref{e:estr conv2}_1$ is verified. If instead  $z_n(0) \to \mu <0$, we argue as follows. Consider $\delta\in \R$ such that $c_n > \delta$, for any $n\in \N$, and let $\eta=\eta(\phi)$ satisfy
\begin{equation}
\label{e:problem mu}
\begin{cases}
\dot{\eta}(\phi)=h(\phi)-\delta-\frac{q(\phi)}{\eta(\phi)},\ \phi > 0,\\
\eta(0)=\mu.
\end{cases}
\end{equation}
By Lemma \ref{lem:cm-dpde} {\em (1.b)} such an $\eta$ exists, in its maximal-existence interval $[0,\sigma)$, for some $\sigma \in (0,1]$. Moreover, we have
$$
\dot{\eta}(\phi)> h(\phi)-c_n -\frac{q(\phi)}{\eta(\phi)}, \ \phi \in (0,\sigma).
$$
Hence, in $(0,\sigma)$, $\eta$ is a strict upper-solution of $\eqref{first order problem0}_1$ with $c=c_n$ and $z_n(0)\le \eta(0)<0$. Thus, Lemma \ref{lem:cm-dpde} {\em (2.a.ii)} implies that $z_n \le \eta$ in $(0,\sigma)$. By passing to the pointwise limit, for $n\to \infty$, it is clear that $\bar{z} \le \eta$ in $(0,\sigma)$. Since $\bar{z}, \eta$ are continuous up to $\phi =0$, then $\bar{z}(0^+) \le \mu$. On the other hand we have $\bar{z}(0^+) \ge \mu$ because $z_n \le \bar{z}$ in $(0,1)$ and $z_n, \bar{z} \in C^0[0,1]$. Then $\bar z(0^+)=\mu$ and this concludes the proof of $\eqref{e:estr conv2}_1$.
\par
Consider $\{z_n\}_n$ decreasing. By adapting the arguments used in the first part of this proof, we can show that $z_n$ converges pointwise in $(0,1)$ to $\bar{z} \in C^0[0,1]\cap C^1(0,1)$ satisfying \eqref{first order problem0}. As before we need \eqref{e:estr conv2} to conclude. To this end, we again observe that similarly to the case of $\{z_n\}_n$ increasing, we have \eqref{e:estr conv2} if both $z_n(0) \to \mu<0$ and $z_n(1) \to \nu <0$. Instead, the proofs of either $\eqref{e:estr conv2}_1$ when $z_n(0)\to 0$ and $\eqref{e:estr conv2}_2$ when $z_n(0) \to 0$ are now more subtle. We provide them both. First, since $z_n <0$ in $(0,1)$, observe that requiring that $z_n(0) \to 0$ (or $z_n(1)\to 0$) corresponds to have $z_n(0)=0$ (or $z_n(1)=0$), for every $n\in\N$.
\par
Take $z_n(0)=0$, for $n\in \N$. Let $n\in \N$ and for $\phi \in (0,1)$, let $\sigma_\phi \in (0,\phi)$ be defined by
$$
\dot{z}_n(\sigma_\phi)=\frac{z_n(\phi)}{\phi}.
$$
Take $\delta_1\in\R$ such that $\delta_1 >c_n$, for each $n\in\N$. By using $\eqref{first order problem0}_1$ and the fact that $q/z_n<0$ in $(0,1)$, we deduce, for any $\phi \in (0,1)$,
\begin{equation}\label{e:zC}
\frac{z_n(\phi)}{\phi}=\dot{z}_n(\sigma_\phi) > h(\sigma_\phi)-c_n > \inf_{\phi \in (0,1)} h(\phi)-\delta_1 =: C <0.
\end{equation}
The sign of $C$ is due to $c_n \ge h(0)$, for $n\in \N$; otherwise, it would not be possible to have $z_n$ satisfying \eqref{first order problem0} and $z_n(0)=0$.
Inequality \eqref{e:zC} implies that $z_n(\phi) > C \phi$ for $\phi \in (0,1)$. Hence, letting $n\to \infty$, this leads to $\bar{z}(\phi) \ge C \phi$, for $\phi \in (0,1)$. Passing to the limit as $\phi \to 0^+$ gives $\bar{z}(0^+) \ge 0$, which in turn implies that $\bar{z}(0^+)=0$. Thus, $\eqref{e:estr conv2}_1$ is verified.
\par
Lastly, let $z_n(1)=0$, for any $n\in \N$. Fix $\epsilon >0$ and consider $\eta_2=\eta_2(\phi)$ such that
\begin{equation}
\label{e:problem mu2}
\begin{cases}
\dot{\eta}_2(\phi)=h(\phi)-\delta-\frac{q(\phi)}{\eta_2(\phi)},\ \phi > 0,\\
\eta_2(1)=-\epsilon <0,
\end{cases}
\end{equation}
where $\delta\in \R$ is such that $\delta < c_n$, for any $n\in \N$. Such an $\eta_2$ exists and is defined and continuous in $[0,1]$, because of Lemma \ref{lem:cm-dpde} {\em (1.a)} and Lemma \ref{lem:zlimit}. Take an arbitrary $n\in\N$. From $0=z_n(1)>\eta_2(1)$, it follows that $\eta_2 < z_n$ in $[{\sigma}_n, 1]$, for some ${\sigma}_n>0$, with $z_n ({\sigma}_n)<0$. Thus, since
$$
\dot{\eta_2}(\phi)>h(\phi)-c_n -\frac{q(\phi)}{\eta_2(\phi)}, \ \phi \in (0,1),
$$
then $\eta_2$ is a strict upper-solution of $\eqref{first order problem0}_1$ with $c=c_n$ in $(0,{\sigma}_n)$ and $\eta_2(\sigma_n) < z_n(\sigma_n) < 0$. An application of Lemma \ref{lem:cm-dpde} {\em (2.a.i)} implies that $\eta_2<z_n$ in $(0,{\sigma}_n)$. Thus, $z_n > \eta_2$ in $(0,1)$, for any $n\in\N$. By passing to the pointwise limit, as $n\to \infty$, we then have $\bar{z}(\phi) \ge \eta_2(\phi)$, for $\phi \in (0,1)$. By the continuity of both $\bar{z}$ and $\eta_2$ at $\phi = 1$, we obtain $0\ge \bar{z}(1^-) \ge -\epsilon$. Since $\epsilon>0$ is arbitrary, we deduce that necessarily $\bar{z}(1^-)=0$.
\end{proof}

Because of Lemmas \ref{lem:zlimit} and \ref{lem:convergence}, in the following we always mean solutions $z$ to problem \eqref{first order problem0}, and analogous ones, in the class $C[0,1] \cap C^1(0,1)$, {\em without any further mention}.

Motivated by Lemma \ref{lem:zlimit}, in the next sections we focus the following problem, where the boundary condition is given on the {\em left} extremum of the interval of definition:
\begin{equation}
\label{first order problem}
\begin{cases}
\dot{z}(\varphi)=h(\varphi)-c-\frac{q(\varphi)}{z(\varphi)}, \ \varphi\in (0,1),\\
z(\varphi) < 0 , \ \varphi \in (0,1),\\
z(0)=0.
\end{cases}
\end{equation}
Problem \eqref{first order problem} is exploited for {\em semi-wavefronts}. The value of $z(1)$ is not prescribed; from $\eqref{first order problem}_2$, we have $z(1) \le  0$. The extremal case $z(1)=0$ is needed in the study of {\em wavefronts}:
\begin{equation}
\label{first order problem 0-0}
\begin{cases}
\dot{z}(\varphi)=h(\varphi)-c-\frac{q(\varphi)}{z(\varphi)}, \ \varphi\in (0,1),\\
z(\varphi) < 0 , \ \varphi \in (0,1),\\
z(0)=z(1)=0.
\end{cases}
\end{equation}

\section{The singular problem with two boundary conditions}\label{s:sing2}
\setcounter{equation}{0}
Problems \eqref{first order problem} and \eqref{first order problem 0-0} have solutions only when $c$ is larger than a critical threshold $c^*$. In this section we first give a new estimate to $c^*$ under mild conditions on $q$; then, we obtain a result of existence and uniqueness of solutions to \eqref{first order problem 0-0} if $c\ge c^*$.
Recalling (D1), (g0) and \eqref{Dg} and (D0)-(g01), throughout the next sections we need to strengthen the assumptions \eqref{e:q} of Section \ref{sec:first order 1}; for commodity we gather them all here below. We assume
\begin{itemize}
\item[{(q)}] $q \in C^0[0,1]$,  $q>0$ in $(0,1)$, $q(0)=q(1)=0$, and $\ds\limsup_{\phi \to 0^+}\frac{q(\phi)}{\phi} < +\infty$.
\end{itemize}
\noindent


\smallskip
We improve, as in \cite[Theorem 3.1]{Marcelli-Papalini}, a well-known result \cite{Aronson-Weinberger, GK, Malaguti-Marcelli_2002}. If $q$ is differentiable at $0$, in
\cite[Theorem 3.1]{Marcelli-Papalini} it is proved that Problem \eqref{first order problem} has a solution if
\begin{equation}
\label{e:c*MP}
c>\sup_{\phi\in(0,1]} \frac{f(\phi)}{\phi} + 2 \sqrt{\sup_{\phi \in (0,1]} \frac1{\phi}\int_{0}^{\phi}\frac{q(\sigma)}{\sigma}\,d\sigma}.
\end{equation}
The last assumption in (q) is weaker than the differentiability of $q$ at $0$ and our result below is less stronger than the one in \cite{Marcelli-Papalini}. It is an open problem whether the existence of solutions to Problem \eqref{first order problem 0-0} under \eqref{e:c*MP} can be achieved by only assuming $\limsup_{\phi \to 0^+}q(\phi)/\phi < +\infty$.

\begin{lemma}
\label{lem:new above est c*}
Assume {\em (q)}. Then Problem \eqref{first order problem 0-0} admits a solution if
\begin{equation}
\label{e:new est c*}
c > \sup_{\phi \in (0,1]} \frac{f(\phi)}{\phi} + 2 \sqrt{\sup_{\phi \in (0,1]} \frac{q(\phi)}{\phi}}.
\end{equation}
\end{lemma}

\begin{proof}
We follow \cite[Theorem 3.1]{Marcelli-Papalini}. By \eqref{e:new est c*} we see that there exists $K>0$, $\eps>0$ so that
\begin{equation*}
\label{equazione8}
K^2 + \left(\sup_{\phi \in (0,1]}\frac{f(\phi)}{\phi}-c\right)K +\sup_{\phi \in (0,1]}\frac{q(\phi)}{\phi} < -\eps K <0 \ \mbox{ for } \ \phi \in (0,1].
\end{equation*}
For every $\tau>0$, we get, for any $\phi>\tau$,
\[
\frac1{\phi - \tau} \int_{\tau}^{\phi} \frac{q(s)}{s}\,ds = \frac{q(s_{\phi,\tau})}{s_{\phi,\tau}}\le \sup_{\phi \in (0,1]}\frac{q(\phi)}{\phi},
\]
where $s_{\phi,\tau}\in (\tau,\phi)$ is given by the Mean Value Theorem. As a consequence, for any $\tau>0$,
\begin{equation*}
\label{equazione9}
K^2+\left(\sup_{\phi\in (0,1]}\frac{f(\phi)}{\phi} +\eps -c\right)K +\frac1{\phi - \tau} \int_{\tau}^{\phi} \frac{q(s)}{s}\,ds < 0\ \mbox{ for every } \ \phi\in(\tau,1].
\end{equation*}
A continuity argument in \cite{Marcelli-Papalini} implies that there exists $\overline{\tau}$ such that for any $\tau < \overline{\tau}$ we have
\[
\frac{f(\phi)-f({\tau})}{\phi - \tau} \le \frac{f(\phi)}{\phi} + \eps \le \sup_{\phi \in (0,1]} \frac{f(\phi)}{\phi} + \eps,\ \phi \in (\tau,1],
\]
and thus, for such values of $\tau$, it must hold
\begin{equation*}
\label{eq2}
K^2+\left(\frac{f(\phi)-f(\tau)}{\phi-\tau} -c\right)K +\frac1{\phi - \tau} \int_{\tau}^{\phi} \frac{q(s)}{s}\,ds<0 \ \mbox{ for every } \ \phi\in(\tau,1].
\end{equation*}
This implies that the function $\eta_\tau=\eta_\tau(\phi)$, defined for $\phi \in [\tau,1]$ by
\begin{equation*}
\label{eq3}
 \eta_\tau(\phi):= -K\tau + \int_{\tau}^{\phi} \left\{h(\sigma)-c-\frac{q(\sigma)}{-K\sigma}\right\}\,d\sigma,
\end{equation*}
 is an upper-solution of $\eqref{first order problem}_1$ such that $\eta_\tau(\phi)< -K\phi$, for $\phi \in (\tau,1]$, and $\eta_\tau(\tau)=-K\tau<0$. Arguments based on Lemma \ref{lem:cm-dpde} {\em (2.a.ii)} imply that it results defined in $[\tau, 1]$ a function $z_\tau$ which solves $\eqref{e:sol from sigma}_2$ with $\mu=-K\tau$; we extend continuously $z_\tau$ to $[0,\tau]$ by $z_\tau(\phi)=-K\phi$, for $\phi \in [0,\tau]$. This gives a family $\{z_\tau\}_{\tau>0}$ of decreasing functions as $\tau \to 0^+$ (in the sense that $z_{\tau_1} \le z_{\tau_2}$ in $[0,1]$ for $0<\tau_1< \tau_2$).
After some manipulations of the differential equation in $\eqref{e:sol from sigma}_2$, based on the sign of $q/z_{\tau}$ and on $\eta_\tau(\phi)< -K\phi$, for $\phi \in (\tau,1]$, we deduce that
\begin{equation*}
\label{equazione10}
f(\phi)-c\phi \le z_\tau(\phi) \le -K\phi, \ \phi \in [0,1].
\end{equation*}
Hence, applying Lemma \ref{lem:convergence} in each interval $(a,b)\subset [0,1]$ we finally deduce that $\bar{z}$, the limit of $z_\tau$ for $\tau \to 0^+$, solves $\eqref{first order problem}_1$, $\bar{z}<0$ in $(0,1)$ and $\bar{z}(0)=0$. Hence, $\bar{z}$ is a solution of \eqref{first order problem}. Finally, as observed in \cite{Marcelli-Papalini}, an application of \cite[Lemma 2.1]{MMM2010} implies the conclusion.
\end{proof}

We now give a result about solutions to \eqref{first order problem 0-0}; see Figure \ref{f:llab} on the left.

\begin{proposition}
\label{prop: first order}
 Assume {\em (q)}. Then, there exists $c^*$ satisfying
\begin{equation}
\label{estimates on c*}
h(0)+2\sqrt{\liminf_{\varphi \to 0^+} \frac{q(\varphi)}{\varphi}} \leq c^* \leq 2\sqrt{\sup_{\varphi \in (0,1]} \frac{q(\varphi)}{\varphi}}+ \sup_{\varphi \in (0,1]} \frac{f(\phi)}{\phi},
\end{equation}
such that there exists a unique $z$ satisfying \eqref{first order problem 0-0} if and only if $c\geq c^*$.
\end{proposition}

\begin{proof}
The result, apart from the refined estimate \eqref{estimates on c*} is proved in \cite[Proposition 1]{MMM2010}. Estimate \eqref{estimates on c*} follows from Lemma \ref{lem:new above est c*} and $\sup_{\phi\in(0,1]}f(\phi)/\phi\le \max_{\phi\in[0,1]}h(\phi)$.
\end{proof}

\section{The singular problem with left boundary condition}\label{s:sing1}
\setcounter{equation}{0}

Now we face problem \eqref{first order problem}.
We always assume (q) and refer to the threshold $c^*$ introduced in Proposition \ref{prop: first order}; we denote by $z^*$ the corresponding unique solution to \eqref{first order problem 0-0}. See Figure \ref{f:llab} on the left for an illustration of Proposition \ref{prop: first order ii}.

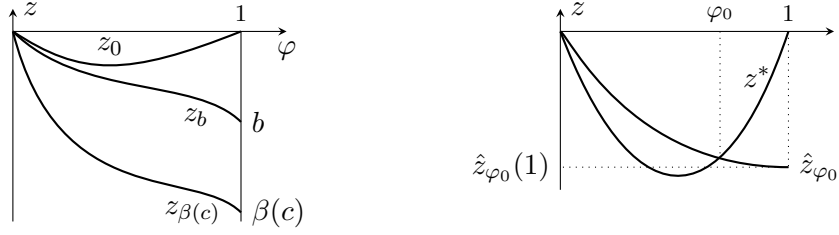
\begin{figure}[htb]
\begin{center}
\begin{tikzpicture}[>=stealth, scale=0.6]
\draw[->] (0,0) --  (6,0) node[below]{$\phi$} coordinate (x axis);
\draw[->] (0,0) -- (0,0.5) node[right]{$z$} coordinate (y axis);
\draw (0,0) -- (0,-4.2);
\draw (5,0) -- (5,-4.2);
\draw[thick] (0,0) .. controls (1.5,-1) and (2.5,-1) .. node[midway, above]{$z_0$} (5,0) ; 
\draw[thick] (0,0) .. controls (1.5,-1.5) and (4,-1) .. node[near end, below]{$z_b$} (5,-2) node[right]{$b$}; 
\draw[thick] (0,0) .. controls (1,-4) and (4,-3) .. node[near end, below]{$z_{\beta(c)}$} (5,-4) node[right]{$\beta(c)$}; 
\draw (5,0)  node[above]{\footnotesize{$1$}};

\begin{scope}[xshift=12cm]
\draw[->] (0,0) --  (6,0) coordinate (x axis);
\draw[->] (0,0) -- (0,0.5) node[right]{$z$} coordinate (y axis);
\draw (0,0) -- (0,-3.5);
\draw[thick] (0,0) .. controls (1.5,-4) and (3.5,-4.5) .. node[very near end, above]{$z^*\ $}  (5,0); 
\draw[thick] (0, 0) .. controls (0.2,-0.1) and (1.5,-3) .. (5,-3) node[right]{$\hat{z}_{\phi_0}$};

\draw (5,0)  node[above]{\footnotesize{$1$}};
\draw[dotted] (3.5,0) node[above]{\footnotesize{$\phi_0$}} -- (3.5,-2.8);
\draw[dotted] (5,-3) -- (0, -3) node[left]{$\hat{z}_{\phi_0}(1)$};
\draw[dotted] (5,0) -- (5,-3);
\end{scope}
\end{tikzpicture}
\end{center}

\caption{\label{f:llab}{Left: an illustration of Propositions \ref{prop: first order} and \ref{prop: first order ii}, for fixed $c>c^*$. Solutions to \eqref{first order problem} are labelled according to their right-hand limit: $z_0$ occurs in the former proposition, $z_b$ in the latter. Right: the functions $\hat{z}_{\phi_0}$ and $z^*$ in {\em Step (i)} of Proposition \ref{prop: first order ii}.}}
\end{figure}

\begin{proposition}
\label{prop: first order ii}
Assume {\em (q)}. For every $c>c^*$, there exists $\beta=\beta(c) <0$ satisfying
\begin{equation}
\label{e:estimate beta}
\beta\ge f(1)-c,
\end{equation}
such that problem \eqref{first order problem} with the additional condition $z(1)=b<0$ admits a unique solution $z$ if and only if $b\geq \beta$.
\end{proposition}

In the above proposition, the threshold case $c=c^*$ is a bit more technical; we shall prove in Proposition \ref{prop: first order 3} that $\beta(c^*)=0$ under some further assumptions.

\begin{proofof}{Proposition \ref{prop: first order ii}}
For any $c>c^*$, we define the set $\mathcal{A}_c$ as
$$
\mathcal{A}_c:=\{b < 0: \eqref{first order problem} \ \mbox{ admits a solution with }\ z(1)=b\}.
$$
We show that $\mathcal{A}_c=[\beta, 0)$, for some $\beta=\beta(c) <0$, by dividing the proof into four steps.

\medskip

{\em Step (i): $\mathcal{A}_c\neq \varnothing$.} We claim that there exists $\hat{z}$ which satisfies \eqref{first order problem} and $\hat{z}(1)<0$. Take $\phi_0\in (0,1)$ and consider the following problem, see Figure \ref{f:llab} on the right,
\begin{equation}
\label{Ac not empty}
\begin{cases}
\dot{z}(\phi)=h(\phi) - c- \frac{q(\phi)}{z(\phi)},\\
z(\phi_0)=z^*(\phi_0).
\end{cases}
\end{equation}
Lemma \ref{lem:cm-dpde} {\em (1)} implies the existence of a solution $\hat{z}_{\phi_0}$ of \eqref{Ac not empty} defined in its maximal-existence interval $(0,\delta)$, for some $\phi_0 < \delta \leq 1$. Since $\hat{z}_{\phi_0}$ satisfies $\eqref{Ac not empty}_1$ and $c>c^*$, then
$$
\dot{\hat{z}}_{\phi_0}(\phi)= h(\phi)-c^* - \frac{q(\phi)}{\hat{z}_{\phi_0}(\phi)} + (c^*-c) < h(\phi)-c^* - \frac{q(\phi)}{\hat{z}_{\phi_0}(\phi)}, \quad \phi\in (0, \delta).
$$
This implies that $\hat{z}_{\phi_0}$ is a strict lower-solution of $\eqref{first order problem}_1$ with $c=c^*$. From Lemma \ref{lem:cm-dpde} {\em (2.b)}, this and $\hat{z}_{\phi_0}(\phi_0)=z^*(\phi_0)<0$ imply that
\begin{equation}\label{e:asd}
z^* <\hat{z}_{\phi_0}\  \mbox{ in } \ (0,\phi_0) \quad \mbox{ and } \quad \hat{z}_{\phi_0} < z^* \ \mbox{ in } \ (\phi_0, \delta).
\end{equation}
Since $z^* < \hat{z}_{\phi_0} <0$ in $(0,\phi_0)$, we get $\hat{z}_{\phi_0}(0^+)=0$. Since $\hat{z}_{\phi_0} < z^*$ in $(\phi_0,\delta)$, we obtain that $\hat{z}_{\phi_0}(\delta^-)\leq z^*(\delta^-)$. Thus $\delta=1$, otherwise $\hat{z}_{\phi_0}(\delta)<0$, in contradiction with the fact that $(0,\delta)$ is the maximal-existence interval of $\hat{z}_{\phi_0}$.

From Lemma \ref{lem:zlimit}, $\hat{z}_{\phi_0}(1)\in\R$. It remains to prove that $\hat{z}_{\phi_0}(1)<0$. From what we observed above, it follows that $z^*>\hat{z}_{\phi_0}$ in $(\phi_0,1)$. Hence, for any $\phi \in (\phi_0,1)$, we have
\[
\dot{z^*}(\phi)-\dot{\hat{z}}_{\phi_0}(\phi) =c-c^* +\frac{q(\phi)}{z^*(\phi)\hat{z}_{\phi_0}(\phi)}\left(z^*-\hat{z}_{\phi_0}\right)(\phi)>\frac{q(\phi)}{z^*(\phi)\hat{z}_{\phi_0}(\phi)}\left(z^*-\hat{z}_{\phi_0}\right)(\phi)>0.
\]
This implies that $(z^*-\hat{z}_{\phi_0})$ is strictly increasing in $(\phi_0, 1)$ and hence
\[
-\hat{z}_{\phi_0}(1)=z^*(1) - \hat{z}_{\phi_0}(1) > z^*(\phi_0) - \hat{z}_{\phi_0}(\phi_0)=0,
\]
which means $\hat{z}_{\phi_0}(1)<0$. Thus, $\hat{z}_{\phi_0}(1)\in \mathcal{A}_c$.

\medskip

{\em Step (ii): if $b \in \mathcal{A}_c$ then $[b,0)\subset \mathcal{A}_c$}. Suppose that there exists $ b \in \mathcal{A}_c$ and let $z_b$ be the solution of \eqref{first order problem} and $z_b(1)=b$. Take $b<b_1 < 0$. For Lemma \ref{lem:cm-dpde} {\em (1.a)} there exists $z_{b_1}$ defined in $(0,1)$ satisfying $\eqref{first order problem}_1$ and $z_{b_1}(1)=b_1<0$.
\par
 We claim that $z_b<z_{b_1}$ in $(0,1)$. If not, then $z_b(\phi_0)=z_{b_1}(\phi_0)=:y_0<0$, for some $\phi_0 \in (0,1)$. Without loss of generality we can assume $z_b<z_{b_1}$ in $(\phi_0, 1]$. We denote by $f_c(\varphi, y)= h(\phi) - c -q(\phi)/y$ the right-hand side of the differential equation in \eqref{first order problem}; the function $f_c$ is continuous in $[0,1]\times(-\infty,0)$ and locally Lipschitz-continuous in $y$. Hence, $z_{b}$ and $z_{b\rq{}}$ are two different solutions of
\begin{equation*}
\begin{cases}
y\rq{}=f_c(\varphi, y), \ \varphi\in(\varphi_0,1),\\
y(\varphi_0)=y_0,
\end{cases}
\end{equation*}
which contradicts the uniqueness of the Cauchy problem. Thus, $z_b<z_{b_1}<0$ in $(0,1)$. Since $z_b$ satisfies $\eqref{first order problem}_3$ then $z_{b_1}(0^+)=0$ and hence $b_1 \in \mathcal{A}_c$.
 \medskip

{\em Step (iii): $\inf \mathcal{A}_c \in \R$}. Suppose that $z$ satisfies Equation $\eqref{first order problem}_1$. As already observed, this implies
$\dot{z}(\phi) > h(\phi)-c$, $\phi \in (0,1)$. Thus, for any $\phi \in (0,1)$,
\begin{equation}
\label{e:bound beta}
z(\phi)=z(\phi)-z(0) \geq \int_{0}^{\phi} h(\sigma) - c\,d\sigma=f(\phi)-c\phi.
\end{equation}
This implies that $z(1) \geq f(1)- c$. Define $\beta=\beta(c)$ by
$$
\beta:=\inf \mathcal{A}_c.
$$
Thus, $\beta \geq   f(1)-c>-\infty$, which also proves \eqref{e:estimate beta}.
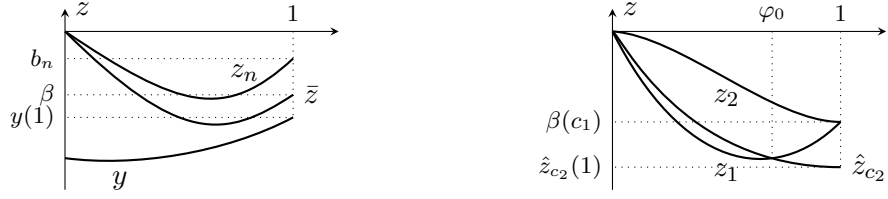
\begin{figure}[htb]
\begin{center}
\begin{tikzpicture}[>=stealth, scale=0.6]
\draw[->] (0,0) --  (6,0) coordinate (x axis);
\draw[->] (0,0) -- (0,0.5) node[right]{$z$} coordinate (y axis);
\draw (0,0) -- (0,-3.5);
\draw[thick] (0,-2.8) .. controls (1.5,-3) and (3.5,-2.7) .. node[near start, below]{$y$}  (5,-1.9); 
\draw[thick] (0,0) .. controls (2.5,-1.7) and (3.5,-2) .. node[above, near end]{$z_n$} (5,-0.6);
\draw[thick] (0,0) .. controls (2.5,-2.5) and (3.5,-2.4) .. (5,-1.4) node[right]{$\bar{z}$} ;
\draw (5,0)  node[above]{\footnotesize{$1$}};
\draw[dotted] (5,0) -- (5,-1.9);
\draw[dotted] (5,-0.6) -- (0,-0.6) node[left]{\footnotesize{$b_n$}};
\draw[dotted] (5,-1.4) -- (0,-1.4) node[left]{\footnotesize{$\beta$}};
\draw[dotted] (5, -1.9) -- (0, -1.9) node[left]{\footnotesize{$y(1)$}};

\begin{scope}[xshift=12cm]
\draw[->] (0,0) --  (6,0) coordinate (x axis);
\draw[->] (0,0) -- (0,0.5) node[right]{$z$} coordinate (y axis);
\draw (0,0) -- (0,-3.5);
\draw[thick] (0,0) .. controls (1.5,-3) and (3.5,-3.5) .. node[below]{$z_1$}  (5,-2);
\draw[thick] (0,0) .. controls (1.5,-0.1) and (3.5,-2) .. node[below]{$z_2$}  (5,-2);
\draw[thick] (0, 0) .. controls (0.2,-0.1) and (1.5,-3) .. (5,-3) node[right]{$\hat{z}_{c_2}$};

\draw (5,0)  node[above]{\footnotesize{$1$}};
\draw[dotted] (3.5,0) node[above]{\footnotesize{$\phi_0$}} -- (3.5,-2.8);
\draw[dotted] (5,-3) -- (0, -3) node[left]{\footnotesize{$\hat{z}_{c_2}(1)$}};
\draw[dotted] (5,-2) -- (0, -2) node[left]{\footnotesize{$\beta(c_1)$}};
\draw[dotted] (5,0) -- (5,-3);
\end{scope}
\end{tikzpicture}
\end{center}

\caption{\label{f:figurestep3}{Left: the functions $z_n$, $y$ and $\bar{z}$ in {\em Step (iv)} of Proposition \ref{prop: first order ii}. Right: the functions $z_1$, $z_2$ and $\hat{z}_{c_2}$ in the proof of {\em (i)} of Corollary \ref{cor:beta}.}}
\end{figure}

\medskip

{\em Step (iv): $\beta\in\mathcal{A}_c$}.
Let $\{b_n\}_n\subset \mathcal{A}_c$ be a strictly decreasing sequence such that $b_n \to \beta^+$. Since $b_n \in \mathcal{A}_c$, each  $b_n$ is associated with a solution $z_n$ of \eqref{first order problem} and $z_n(1)=b_n$. From the uniqueness of the solution of Cauchy problem for $\eqref{first order problem}_1$, the sequence $z_n$ is decreasing.

For any given $\delta<\beta$, let $y$ be defined by
\begin{equation*}
\begin{cases}
\dot{y}(\phi)=h(\phi)-c -\frac{q(\phi)}{y(\phi)}, \ \phi< 1\\
y(1)=\delta< \beta.
\end{cases}
\end{equation*}
Such a $y$ exists and is defined in $[0,1]$ from Lemma \ref{lem:cm-dpde} {\em (1.a)}. Also, $b_n > \delta$, for any $n\in \N$. Thus, for any $n\in\N$, $z_n \ge y$ in $[0,1]$. Lemma \ref{lem:convergence} implies that there exists $\bar{z}$ satisfying \eqref{first order problem0} such that $z_n \to \bar{z}$ uniformly in $[0,1]$ (see Figure \ref{f:figurestep3} on the left). In particular, we deduce that $\bar{z}(0)=0$ and $\bar{z}(1)=\beta$. Hence, we conclude that $\beta \in \mathcal{A}_c$.

\par
Putting together {\em Steps (i) -- (iv)}, we conclude that $\mathcal{A}_c=[\beta,0)$.
\end{proofof}

The monotonicity of solutions of \eqref{first order problem} now follows. We omit the proof since it is quite standard, once that Lemma \ref{lem:cm-dpde} {\em (2)} is given. (See \cite[Lemma 5.1]{CM-DPDE}.)

\begin{corollary}[Monotonicity of solutions]
\label{cor:monotonicity}
Assume {\em (q)}. Let $c_2>c_1\geq c^*$ and assume that $z_1$ and $z_2$ satisfy \eqref{first order problem} with $c=c_1$ and $c=c_2$, respectively. Then, if $z_1(1)\leq z_2(1)$ it occurs that $z_1<z_2$ in $(0,1)$.
\end{corollary}
%

A monotony property of $\beta(c)$ now follows.

\begin{corollary}
\label{cor:beta}
Under {\em (q)} we have:
\begin{enumerate}[(i)]

\item $\beta(c_2) < \beta(c_1)$ for every $c_2 > c_1 > c^*$;

\item $\beta(c) \to -\infty$ as $c \to +\infty$.

\end{enumerate}
\end{corollary}
\begin{proof}
To prove {\em (i)}, let $z_1$ be a solution of \eqref{first order problem} corresponding to $c=c_1$ and such that $z_1(1)=b_1\in\mathcal{A}_{c_1}$. As a consequence of Lemma \ref{lem:cm-dpde} {\em (1.a)}, the problem
\begin{equation*}
\begin{cases}
\dot{z}(\phi)=h(\phi)-c_2 - \frac{q(\phi)}{z(\phi)}, \ \phi \in (0,1),\\
z(1)=b_1<0,
\end{cases}
\end{equation*}
admits a (unique) solution $z_2$ defined in $[0,1]$. From the monotonicity of solutions given by Corollary \ref{cor:monotonicity}, we have $z_1 < z_2<0$ in $(0,1)$.
Since $z_1(0)=0$, then we have $z_2(0)=0$. Thus, $\mathcal{A}_{c_1}\subseteq \mathcal{A}_{c_2}$ and hence $\beta(c_1)\geq \beta(c_2)$. To prove $\beta(c_1)> \beta(c_2)$ we argue as follows.
\par
For any $\phi_0 \in (0,1)$ we can repeat the same arguments as in {\em Step (i)} of Proposition \ref{prop: first order ii}, by replacing $c$ with $c_2$ and $z^*$ with $z_1$ in \eqref{Ac not empty}. Thus, the problem
\begin{equation*}
\begin{cases}
\dot{z}(\phi)=h(\phi) - c_2 - \frac{q(\phi)}{z(\phi)}, \ \phi\in (0,1),\\
z(\phi_0)=z_1(\phi_0) < 0,
\end{cases}
\end{equation*}
admits a unique solution $\hat{z}_{c_2}$ defined in $[0,1]$, because necessarily any solution of the last problem must be bounded from above by $z_2$, see Figure \ref{f:figurestep3} on the right.. Moreover, by applying Lemma \ref{lem:cm-dpde} {\em (2.b.ii)}, $\hat{z}_{c_2} <z_1$ in $(\phi_0,1)$, which implies that $\hat{z}_{c_2}(1)<z_1(1)$, since
$$
\dot{\hat{z}}_{c_2}(\phi)-\dot{z_1}(\phi)=c_1-c_2 +\frac{q(\phi)}{z_1(\phi)\hat{z}_{c_2}(\phi)}\left(\hat{z}_{c_2}(\phi)-z_1(\phi)\right) <0 \ \mbox{ for any } \ \phi \in (\phi_0, 1).
$$
Since $\beta(c_2)\leq \hat{z}_{c_2}(1)<z_1(1)=b_1$ then we proved {\em (i)} since $b_1$ is arbitrary in $\mathcal{A}_{c_1}$.
\par
Finally, we prove {\em (ii)}. For $c>c^*$, let $z_c$ be the solution of \eqref{first order problem} such that $z_c(1)=\beta(c)$. For any fixed $c_1>c^*$, we have $z_c < z_{c_1}$ in $(0,1)$, if $c>c_1$. Thus, for any $c>c_1$,
$$
\dot{z}_c(\phi) = h(\phi)-c + \frac{q(\phi)}{-z_c(\phi)} < h(\phi)-c +\frac{q(\phi)}{-z_{c_1}(\phi)}, \ \phi \in (0,1).
$$
In particular, since $z_{c_1}<0$ in $(0,1]$, then, for any $0<\delta < 1$, there exists $M>0$ such that $q(\phi)/(-z_{c_1}(\phi)) \leq M$ for any $\phi \in (\delta,1]$. Thus, for any $\phi \in (\delta, 1)$,
$$
z_c(\phi)\leq z_c(\delta) + f(\phi) - f(\delta)+ \left(M - c\right)(\phi - \delta) < f(\phi) - f(\delta) + \left(M-c\right)(\phi - \delta),
$$
which implies $\beta(c)=z_c(1) \leq f(1) - f(\delta) + (M-c)(1-\delta)$. This proves {\em {(ii)}}.
\end{proof}

We now collect some consequences of \eqref{e:bound beta} and Lemma \ref{lem:new above est c*}, concerning a sharper estimate to $c^*$. To the best of our knowledge these estimates are new and we provide some comments.

\begin{corollary}
\label{cor:c*}
Assume {\em (q)}. It holds that
\begin{equation}
\label{e:c*1}
c^*\ge
\max\left\{\sup_{\phi \in (0,1]} \frac{f(\phi)}{\phi}, h(0) + 2\sqrt{\liminf_{\phi \to 0^+} \frac{q(\phi)}{\phi}} \right\}.
\end{equation}
\end{corollary}

\begin{proof}
Formula \eqref{e:bound beta} in {\em Step (iii)} implies that $f(\phi) < c\phi$, for $\phi \in (0,1)$. Thus, $f(\phi)\le c^* \phi$, for $\phi \in (0,1)$. This implies $c^*\ge \sup_{\phi \in (0,1]} \frac{f(\phi)}{\phi}$, which, together with \eqref{estimates on c*} implies \eqref{e:c*1}.
\end{proof}

\begin{remark}
\label{rem:MP}
{ \rm
Lemma \ref{lem:new above est c*} and Corollary \ref{cor:c*} imply that, under (q), the threshold $c^*$ verifies \eqref{estimates on c* new}. Moreover, make the assumption  $\dot{q}(0)=0$, which is valid if $q=Dg$ under (D1), with $D(0)=0$, (g0) or under (D0) and (g01). In this case, the estimates in \eqref{e:est c*3} hold true. Indeed, the assumptions on $q$ are covered by \cite[Theorem 3.1]{Marcelli-Papalini} and hence it follows that
\[
c^* \le \sup_{\phi\in(0,1]} \frac{f(\phi)}{\phi} + 2 \sqrt{\sup_{\phi \in (0,1]} \frac1{\phi}\int_{0}^{\phi}\frac{q(\sigma)}{\sigma}\,d\sigma}.
\]
The bound from above in \eqref{e:est c*3} is then proved. The bound from below in \eqref{e:est c*3} is instead due directly to \eqref{e:c*1}, because of $\dot{q}(0)=0$.
}
\end{remark}

\begin{remark}
\label{r:gap}
{\rm
We can now make precise the statement following formula \eqref{e:zaz} about the gap between $c_{con}$ and $c^*$. If $c_{con}$ is obtained at some $\phi \in (0,1]$, then the $\sup$ in the right-hand side of \eqref{e:zaz} is strictly larger than $c_{con}$ because $z<0$ in $(0,1)$. Then $c^*>c_{con}$. Otherwise, if $\sup_{\phi\in (0,1]}f(\phi)(\phi)=h(0)$, then $c_{con} =h(0)$ and by \eqref{e:c*1} we still deduce $c^* >c_{con}$.
}
\end{remark}
\section{Further existence and non-existence results}
\setcounter{equation}{0}
Propositions \ref{prop: first order} and \ref{prop: first order ii} completely treat the existence of solutions of \eqref{first order problem 0-0} and \eqref{first order problem}, respectively, in the cases $c\ge c^*$ and $c>c^*$. In this section, we investigate the remaining cases and show that such propositions are somehow optimal.

We now deal with the following problem, where $c\in\R$ but, differently from \eqref{first order problem}, the boundary condition is imposed on the {\em right} extremum of the interval of definition:
\begin{equation}
\label{f.o. problem 2}
\begin{cases}
\dot{\zeta}(\phi)=h(\phi)-c- \frac{q(\phi)}{\zeta(\phi)}, \ \phi\in (0,1),\\
\zeta(\phi)<0, \ \phi \in (0,1),\\
\zeta(1)=0.
\end{cases}
\end{equation}
The differential equation in \eqref{first order problem} and \eqref{f.o. problem 2} is the same; it inherits the properties of the dynamical system underlying \eqref{e:ODE}. For slightly more regular functions $g$, the dynamical system has a center or a node at $(0,0)$ and a saddle at $(1,0)$. The corresponding results, Proposition \ref{prop: first order ii} and Lemma \ref{lem:right pb}, differ as in Lemma \ref{lem:cm-dpde} {\em (1)}.

Moreover, while in problem \eqref{first order problem} the threshold $c^*$ discriminated the existence of solutions, for problem \eqref{f.o. problem 2} solutions will be proved to exist {\em for every} $c\in\R$; instead, the threshold $c^*$ enters into the problem to discriminate whether solutions reach $0$ or not (see Figure \ref{f:llab2}). A related behavior was pointed out in \cite[Theorem 2.6]{CM-DPDE}. On the contrary, the monotonicity properties stated in Corollary \ref{cor:monotonicity} and in Lemma \ref{lem:right pb} are the same.

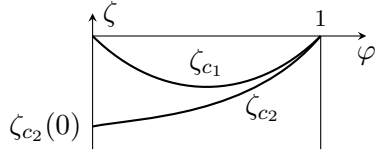
\begin{figure}[htb]
\begin{center}
\begin{tikzpicture}[>=stealth, scale=0.6]
\draw[->] (0,0) --  (6,0) node[below]{$\phi$} coordinate (x axis);
\draw[->] (0,0) -- (0,0.5) node[right]{$\zeta$} coordinate (y axis);
\draw (0,0) -- (0,-2.5);
\draw (5,0) -- (5,-2.5);
\draw[thick] (0,0) .. controls (1.5,-1.5) and (3.5,-1.5) .. node[midway, above]{$\zeta_{c_1}$} (5,0) ; 
\draw[thick]  (0,-2) node[left]{$\zeta_{c_2}(0)$} .. controls (1,-1.8) and (3.5,-1.8) .. node[near end, below]{$\zeta_{c_2}$} (5,0) ; 
\draw (5,0)  node[above]{\footnotesize{$1$}};
\end{tikzpicture}
\end{center}

\caption{\label{f:llab2}{An illustration of Lemma \ref{lem:right pb}. Here, $c_1\geq c^*$ while $c_2<c^*$ and $\zeta_{c_2}(0)<0$.}}
\end{figure}

\begin{lemma}
\label{lem:right pb}
Assume {\em (q)}. For any $c\in \R$, Problem \eqref{f.o. problem 2} admits a unique solution $\zeta_c$. If $c\ge c^*$ then $\zeta_c(0)=0$ and if $c<c^*$ then $\zeta_c(0)<0$. Moreover, we have:
 \begin{enumerate}[(i)]
 \item  if $c_2 >c_1$ then $\zeta_{c_2}>\zeta_{c_1}$ in $(0,1)$;

 \item it holds that $z^*(\phi)=\lim_{c\to c^*} \zeta_{c}(\phi)$ for any $\phi\in [0,1]$.

\end{enumerate}
\end{lemma}

\begin{proof}
The existence and uniqueness was proved in \cite[Theorem 2.6]{CM-DPDE}, while the monotonicity as stated in {\em (i)} was given in \cite[Lemma 5.1]{CM-DPDE}. It remains to prove {\em (ii)}. We show that
$$
\lim_{\delta\to 0^+} \zeta_{c^*-\delta}(\phi)=\lim_{\delta \to 0^+}\zeta_{c^*+ \delta}(\phi)=z^*(\phi) \quad  \hbox{ for } \phi \in [0,1].
$$
For any $\phi \in [0,1]$, by {\em (i)} we have
\begin{equation}
\label{e:zeta chain}
\zeta_{c^*-\delta_2}(\phi) < \zeta_{c^*-\delta_1}(\phi) < z^*(\phi) < \zeta_{c^*+\delta_1}(\phi)<\zeta_{c^*+\delta_2}(\phi) \ \mbox{ for any } \ 0<\delta_1<\delta_2.
\end{equation}
Lemma \ref{lem:convergence} and \eqref{e:zeta chain} imply that there exist two functions $\overline{w}, \underline{w}\in C^0[0,1]\cap C^1\left(0,1\right)$ so that
$\overline{w}(\phi)=\lim_{\delta \to 0^+}\zeta_{c^* + \delta}(\phi)$ and $\underline{w}(\phi)=\lim_{\delta\to 0^+}\zeta_{c^*-\delta}(\phi)$, $\phi \in [0,1]$,
and that both $\underline{w}$ and $\overline{w}$ satisfy \eqref{first order problem0} with $c=c^*$. Since $\underline{w}(1)=\overline{w}(1)=0$, both of them then solve \eqref{f.o. problem 2}. By the uniqueness of solutions of \eqref{f.o. problem 2} it follows that $\underline{w}=\overline{w}=z^*$.
\end{proof}

\begin{remark}
{
\rm
Note that, because of the uniqueness stated in Lemma \ref{lem:right pb}, it follows that, for any $c\ge c^*$, the solution $z$ given by Proposition \ref{prop: first order} corresponds to $\zeta_c$ of Lemma \ref{lem:right pb}. Moreover, for $c<c^*$ fixed, there exists a bound from below for $\zeta_c(0)<0$. We have
$$
\zeta_c(0)\ge -1-  A_c,\quad \hbox{ for }
A_c:=\max\bigl\{\max_{\phi \in [0,1]} h(\phi) -c,0\bigr\}+ \max_{\phi \in [0,1]} q(\phi) >0.
$$
Indeed, the function $\eta(\phi):=A_c\left(\phi - 1\right) -1$, for $\phi \in [0,1]$, is a strict upper-solution of $\eqref{f.o. problem 2}_1$. Therefore, if $\zeta_c(\phi_0)\le \eta(\phi_0)$, for some $\phi_0 \in (0,1)$, then $\zeta_c < \eta$ in $(\phi_0, 1)$ by Lemma \ref{lem:cm-dpde} {\em (2.a.ii)}, which is in contradiction with $\zeta_c(1)=0>\eta(1)$. Thus, $\zeta_c(0)\ge \eta(0) = -A_c-1$. Notice that, for $c \ge \max h$, $A_c=\max q$ does not depend on $c$, while $A_c \to \infty$, as $c\to -\infty$.
}
\end{remark}

We now show that $\beta(c^*)=0$ under some additional conditions. First, we assume (also for future reference) that $\dot{q}(0)$ exists:
\begin{equation}
\label{e:D1g1 at zero}
\dot q(0) = \lim_{\phi \to 0^+} \frac{q(\phi)}{\phi}\in[0,\infty).
\end{equation}

 \begin{proposition}
 \label{prop: first order 3}
Assume {\em (q)}, \eqref{e:D1g1 at zero} and also
\begin{equation}\label{e:c=c^*2e}
\int_{0} \frac{q(\sigma)}{\sigma^2} \, d\sigma< +\infty  \quad \hbox{ and }\quad c^* > h(0).
\end{equation}
Then Problem \eqref{first order problem} with $c=c^*$ admits a unique solution $z$, which satisfies $z(1)=0$.
 \end{proposition}

Notice that $\eqref{e:c=c^*2e}_1$ above strengthens the last condition in (q) and is satisfied if $\dot q(\phi)=O(\phi^\alpha)$ for $\phi\to0^+$, for some $\alpha>0$; in any case it implies $\dot q(0)=0$ by \eqref{e:D1g1 at zero}.

\begin{figure}[htb]
\begin{center}
\begin{tikzpicture}[>=stealth, scale=0.6]
\draw[->] (0,0) --  (6,0) node[below]{$\phi$} coordinate (x axis);
\draw[->] (0,0) -- (0,0.5) node[right]{$z$} coordinate (y axis);
\draw (0,0) -- (0,-3.5);
\draw (5,0) -- (5,-3.5);
\draw[thick] (0,0) .. controls (1.5,-1) and (2.5,-1) .. node[midway, above]{$z^*$} (5,0) ; 
\draw[thick] (0,-2) node[left]{\footnotesize{$\zeta_c(0)$}} .. controls (1.5,-1.8) and (4,-1) .. node[very near end, below]{$\zeta_c$} (5,0); 
\draw[thick] (0,0) .. controls (1.5,-1) and (2.5,-1.8) .. node[near start, below]{$y^*$} (5,-2.2) ; 
\draw[thick] (0,-2) .. controls (1,-2.4) and (2.5,-2.9) .. node[midway, below]{$z_c^*$} (5,-3.3) ; 

\draw (5,0)  node[above]{\footnotesize{$1$}};
\end{tikzpicture}
\end{center}

\caption{\label{f:z^*zeta_c}{The functions $z^*$, $\zeta_c$, $y^*$ and $z_c^*$, for $c<c^*$. }}
\end{figure}
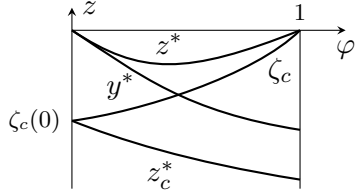
 \begin{proofof}{Proposition \ref{prop: first order 3}}
Suppose, by contradiction, that there exists $y^*$ which solves \eqref{first order problem} with $c=c^*$ and $y^*(1)<0$; observe that
\begin{equation}\label{e:zeta*y*}
z^* > y^*\quad \hbox{ in } (0,1].
\end{equation}
We show that $y^*$ is an upper bound for the family of functions $\{z_c^*\}_{c<c^*}$ defined as follows, see Figure \ref{f:z^*zeta_c}. For any $c<c^*$, let $\zeta_c$ be the solution of \eqref{f.o. problem 2}, given in Lemma \ref{lem:right pb}. Consider the initial-value problem
\begin{equation}
\label{initial-value problem}
\begin{cases}
\dot{z}(\phi)=h(\phi)-c^*-\frac{q(\phi)}{z(\phi)},\ \phi\in(0,1),\\
z(0)=\zeta_c(0)<0.
\end{cases}
\end{equation}
By Lemma \ref{lem:cm-dpde} {\em (1.b)}, problem \eqref{initial-value problem} admits a unique solution $z_c^*$ in $[0, \delta]$ for some $\delta \leq 1$. Moreover, since $z_c^*(0)<0$ and $z_c^*$ satisfies \eqref{initial-value problem}, then $z_c^* < y^*$ in $[0,\delta)$. Thus, if $\delta < 1$ then we have $-\infty<z_c^*(\delta)< y^*(\delta)< 0$; again by Lemma \ref{lem:cm-dpde} {\em (1.b)} we deduce $\delta=1$. Then
\begin{equation}
\label{e:y*}
y^*>z_c^* \ \mbox{ in } \ [0,1).
\end{equation}

By both Lemma \ref{lem:right pb} {\em (ii)} and \eqref{e:y*} we now find a contradiction, which implies that such a $y^*$ cannot exist. For this, for any $c<c^*$, define $\eta_c$ by
$$
\eta_c(\phi)=\zeta_c(\phi) - z_c^*(\phi), \quad \phi\in[0,1].
$$
Since $z_c^*$ is a strict lower-solution of $\eqref{first order problem}_1$, then Lemma \ref{lem:cm-dpde} {\em (2.b.ii)} implies $\eta_c>0$ in $(0,1)$. We claim that, for any fixed $\phi_0\in (0,1]$,  $\eta_c(\phi_0)$ is uniformly bounded from below for $c$ close to $c^*$. Indeed, for any  $0<\delta<(z^*-y^*)(\phi_0)$, we clearly have, by \eqref{e:y*} and \eqref{e:zeta*y*},
\[
\eta_c(\phi_0) > \zeta_c(\phi_0) - y^*(\phi_0) =
\left(\zeta_c - z^*\right)(\phi_0) + \left(z^*- y^*\right)(\phi_0) > \left(\zeta_c - z^*\right)(\phi_0) + \delta.
\]
Thus, in virtue of Lemma \ref{lem:right pb} {\em (ii)}, for any $c$ sufficiently close to $c^*$, we have
\begin{equation}
\label{e:etan 1}
\eta_c(\phi_0) \geq \frac{\delta}{2}>0,
\end{equation}
which proves our claim. On the other hand, define $k=k(\phi)>0$ by
$$
k(\phi):=\frac{q(\phi)}{\left(z^* y^*\right)(\phi)}, \quad \phi \in(0,1).
$$
From assumptions \eqref{e:D1g1 at zero} and $\eqref{e:c=c^*2e}_2$ we deduce $\dot{z}^*(0)=h(0)-c^*<0$ because of \cite[Proposition 5.2]{CM-DPDE}. Also, by \eqref{e:zeta*y*} we deduce that $y^*z^*> {z^*}^2$ in $(0,1]$. Thus,
$$
k(\phi) < \frac{q(\phi)}{\phi^2} \left(\frac{\phi}{z^*(\phi)}\right)^2 = \frac{q(\phi)}{\phi^2} \left\{\frac1{\left(c^* - h(0) \right)^2} + o(1)\right\} \ \mbox{ for } \ \phi \to 0^+.
$$
This leads to
\[
\int_{0}^{\phi_0} k(\sigma)\,d\sigma=:M  < +\infty
\]
by means of $\eqref{e:c=c^*2e}_1$. Since $\zeta_c$ and $z_c^*$ satisfy $\eqref{first order problem}_1$ with $c<c^*$ and $c=c^*$, respectively, and since $\zeta_c z_c^* > z^* y^*$ by the monotonicity stated in Lemma \ref{lem:right pb} and \eqref{e:y*}, then
\[
\dot{\eta_c}(\phi)= c^*-c- \frac{q(\phi)}{\zeta_c(\phi) z_c^*(\phi)}\left(z_c^*(\phi)-\zeta_c(\phi)\right) <
c^*-c + k(\phi)\eta_c(\phi),
\]
for $\phi\in (0,1)$. After some straightforward manipulations, this gives
$$
\frac{d}{d\phi}\left(\eta_c(\phi) e^{-\int_{0}^{\phi} k(\sigma)\,d\sigma}\right) \leq \left(c^*-c\right)e^{-\int_{0}^{\phi} k(\sigma)\,d\sigma}, \ \phi\in (0,1).
$$
By integrating in $(0,\phi_0)$ (where $\phi_0$ is the point for which \eqref{e:etan 1} holds) we obtain
\begin{equation}
\label{e:eta_c v2}
0< \eta_c(\phi_0)\leq  \left(c^* - c\right)e^{\int_{0}^{\phi_0} k(\sigma)\,d\sigma}\int_{0}^{\phi_0} e^{-\int_{0}^{\sigma} k(s)\,ds}\,d\sigma\le \left(c^*-c\right) e^M \phi_0,
\end{equation}
since $e^{-\int_{0}^{\sigma} k(s)\,ds} \le 1$, for any $0<\sigma < \phi_0$, because of $k>0$. Since $M$ does not depend on $c$, from \eqref{e:eta_c v2}, we conclude that $\eta_c(\phi_0) \to 0$, for $c\to c^*$. This contradicts \eqref{e:etan 1}.
 \end{proofof}

We notice that if $q=Dg$, with $D\in C^1[0,1]$, then $\eqref{e:c=c^*2e}_1$ follows if we have both
$D(0)=0$ and there exists $L\ge 0$ such that $g(\phi)\le L \phi^\alpha$ for any $\phi$ in a right neighborhood of $0$ and some $\alpha>0$. The next remark deals with $\eqref{e:c=c^*2e}_2$.

 \begin{remark}
 \label{rem:c*=0}
{
\rm
First, from \eqref{estimates on c*}, we have $c^*\geq \sup_{\phi \in (0,1]} \frac{f(\phi)}{\phi} \ge h(0)$. We show that the case $c^*=h(0)$ can indeed occur and then $\eqref{e:c=c^*2e}_2$ is a real assumption.  Set, for $\phi\in (0,1)$,
\begin{equation}\label{e:qhrem}
q(\phi)=\phi^3\left(1 - \phi\right), \ h(\phi)=3\phi\left(\phi- 1\right),
\end{equation}
and $z(\phi)=\phi^2\left(\phi - 1\right)$. Direct computations show that $z$ satisfies \eqref{first order problem} with $c=0=h(0)$. Hence, $c^*=h(0)$, because of $c^*\geq h(0)$.
\par
Second, in the spirit of \cite[Theorems 1.2 and 1.3]{MMconv}, which concerns a similar case, we claim that $\eqref{e:c=c^*2e}_2$ occurs if there exists $\delta >0$ such that $h(\phi)\geq h(0)$ for all $\phi \in [0,\delta]$.
Indeed, if $z$ is a solution of \eqref{first order problem} with $c=c^*$, then from $\eqref{first order problem}_1$ we have $\dot{z}(\phi) > h(\phi)-c^* \geq h(0)-c^*$, for $\phi \in (0,\delta)$.
This implies $h(0)-c^* \leq \inf_{\phi \in (0,\delta)}\dot{z}(\phi)<0$, because of $\eqref{first order problem}_2$ and $\eqref{first order problem}_3$, which proves our claim.
\par
Lastly, we show by a counter-example that the conclusion of Proposition \ref{prop: first order 3} fails when $\eqref{e:c=c^*2e}_1$ holds but $\eqref{e:c=c^*2e}_2$ {\em does not}. Consider, for $\phi \in [0,1]$,
$q(\phi)=\phi^4\left(1-\phi\right)$ and $y^*(\phi)=-\phi^2$.
Clearly, $y^*<0$ in $(0,1)$ and $y^*(0)=0$. Furthermore, we have
\[
\dot{y}^*(\phi) +\frac{q(\phi)}{y^*(\phi)} = -2\phi -\phi^2\left(1-\phi\right),  \ \phi \in (0,1).
\]
This implies that $y^*$ satisfies $\eqref{first order problem}_1$ with
$h(\phi)=-2\phi -\phi^2\left(1-\phi\right)$ and $c=0$. As a consequence, by $c^*\ge h(0)=0$, we deduce $c^*=h(0)=0$. Thus, we proved that there exists $q$ satisfying $\eqref{e:c=c^*2e}_1$ such that \eqref{first order problem} with $c=c^*=h(0)$ admits a solution $y^*\neq z^*$.
}
\end{remark}

 \begin{proposition}
 \label{prop: first order 4}
 Assume {\em (q)}. For no $c<c^*$ problem \eqref{first order problem} admits solutions.
 \end{proposition}
 \begin{proof}
 Take $c<c^*$ and assume by contradiction that problem \eqref{first order problem} has a solution $z$.  If $\zeta=\zeta_c$ is the solution to \eqref{f.o. problem 2} given by Lemma \ref{lem:right pb}, then $\zeta(0)<0$, by Proposition \ref{prop: first order}. Then $\zeta(\varphi_0)=z(\varphi_0)=:y_0<0$, for some $\varphi_0\in (0,1)$; see Figure \ref{f:fzz2}. This contradicts the uniqueness of the Cauchy problem associated to $\eqref{f.o. problem 2}_1$. The proof is concluded.

\begin{figure}[htb]
\begin{center}
\begin{tikzpicture}[>=stealth, scale=0.6]

\draw[->] (0,0) --  (6,0) node[below]{$\phi$} coordinate (x axis);
\draw[->] (0,0) -- (0,0.5) node[right]{$z$} coordinate (y axis);
\draw (0,0) -- (0,-3.4);
\draw[thick] (0,0) .. controls (2.5,-4) and (3.5,-4) .. node[near end, below]{$z$} (5,-0.8) ; 
\draw[thick] (0,-1.5) node[left]{\footnotesize{$\zeta(0)$}} .. controls (1.5,-2.5) and (3.5,-2.5) .. node[very near start, below]{$\zeta$}  (5,0) ; 

\draw[dotted] (0,-0.8) node[left]{\footnotesize{$z(1)$}}-- (5,-0.8);
\draw (5,0)  node[above]{\footnotesize{$1$}};
\draw[dotted] (1.5,0) node[above]{\footnotesize{$\phi_0$}} -- (1.5,-2.2);
\draw[dotted] (5,0) -- (5,-0.8);

\end{tikzpicture}
\end{center}

\caption{\label{f:fzz2}{The functions $z$ and $\zeta$.}}
\end{figure}
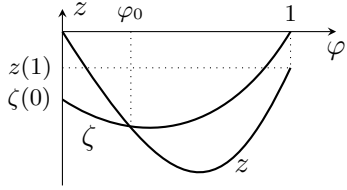

 \end{proof}

 \section{The behavior of $z$ near $1$}
 \label{sec:further properties}
\setcounter{equation}{0}

In this section and in the next one we investigate the behavior of the solutions $z$ to \eqref{first order problem} at $1$ and $0$. We now deal with the former case. We suppose that, analogously to \eqref{e:D1g1 at zero},
\begin{equation}\label{e:dotq-ex-fin}
\dot{q}(1)\in(-\infty,0].
\end{equation}

\begin{proposition}
\label{prop:derivative of z at alpha}
Assume {\em (q)} and \eqref{e:dotq-ex-fin}; consider $c\geq c^*$ and let $z$ be a solution of \eqref{first order problem}. Then, $\dot{z}(1)$ exists and it holds that
\begin{enumerate}[(i)]
\item if $z(1) \in [\beta, 0)$, then
$\dot{z}(1)=h(1)-c$;
\item if $z(1)=0$, then
\begin{equation}
\label{derivative of z at alpha}
\dot{z}(1) =
\left\{
\begin{array}{ll}
\frac1{2}\left[ h(1) - c + \sqrt{\left(h(1) - c\right)^2 - 4\dot{q}(1)} \right]
& \hbox{ if }\dot{q}(1)<0,
\\
\max\left\{0,h(1)-c\right\} & \hbox{ if }\dot{q}(1)=0.
\end{array}
\right.
\end{equation}
\end{enumerate}
\end{proposition}

\begin{proof}
In case {\em (i)}, we only need to take the limit for $\phi\to 1^-$ in $\eqref{first order problem}_1$.

In case {\em (ii)}, the proof of the existence of $\dot z(1)$ is analogous to the proof of \cite[Lemma 2.1]{MMconv}, even if in that paper there is the further assumption $\dot q (1)=0$. In the current case, $\dot z(1)$ must coincide with one of the roots of the equation
$\gamma^2 - \left(h(1) - c\right) \gamma + \dot{q}(1)=0$,
which are
$$
r_\pm:= \frac{h(1)-c \pm \sqrt{\left(h(1)-c\right)^2 - 4 \dot{q}(1)}}{2}.
$$
A direct check shows that the right-hand side of \eqref{derivative of z at alpha} corresponds exactly to $r_+$. Thus, if we prove that $\dot z(1) = r_+$ then we conclude the proof.

If $\dot{q}(1)<0$, the fact that $r_ -<0$ implies necessarily that $\mu=r_+$, because of $\dot z(1)\geq 0$.

Let $\dot{q}(1)=0$. Since we do not yet know whether $\dot z$ is continuous at $1$ (see Remark \ref{rem:z c1}), we argue as follows. For any $\phi \in (0,1)$, by the Mean Value Theorem there exists $\sigma_\phi \in (\phi, 1)$ satisfying
$\dot{z}(\sigma_\phi)=\frac{z(\phi)}{\phi-1}$.
By the definition of $\dot{z}(1)$ it then follows that
\begin{equation}
\label{e:sigmaphi}
\lim_{\phi \to 1^-}\dot{z}(\sigma_\phi)=\dot z(1) \ \mbox{ and } \  \lim_{\phi\to 1^-} \frac{z(\sigma_\phi)}{\sigma_\phi - 1}=\dot z(1).
\end{equation}
From $\eqref{first order problem}_1$, the sign conditions in $\eqref{e:q}_2$ and $\eqref{first order problem}_2$ imply that
\begin{equation}
\label{e:dotzsigmaphi}
\dot{z}(\sigma_\phi) > h(\sigma_\phi) - c, \quad \phi \in (0,1).
\end{equation}
By \eqref{e:sigmaphi}, passing to the limit as $\phi\to 1^-$ gives $\dot z(1) \geq h(1)-c$, because of the continuity of $h$ at $1$. Moreover, since $\dot z(1) \geq 0$ it holds that $\dot z(1) \geq \max\left\{0, h(1)-c\right\}=r_+$.
This concludes the proof, since it necessarily follows that $\dot z(1)=r_+$ also in this case.
\end{proof}

\begin{remark}
{
\rm
We prove in Remark \ref{rem:z c1} that $z\in C^1(0,1]$ under the assumptions of Proposition \ref{prop:derivative of z at alpha}. We now show that \eqref{e:dotq-ex-fin} is necessary for the existence of $\dot z(1)$. We define
\[
q(\phi)=\phi^3\left(1-\phi \right)\left[\left(\sin\left(\log\left(1-\phi\right)\right)+2\right)^2 +2\cos\left(\log\left(1-\phi\right)\right)+\frac1{2}\sin\left(2\log\left(1-\phi\right)\right)\right],
\]
for $\phi \in [0,1]$. The function $q$ satisfies (q), while $\dot{q}(1)$ does not exist. Direct computations show that the function $z=z(\phi)$ defined by $z(\phi)=-\left(2+\sin(\log(1-\phi))\right)\left(1-\phi\right)\phi^2$
satisfies \eqref{first order problem} with $c=0$ and
$h(\phi)=\phi(\phi-1)\left[\cos(\log(1-\phi))+3\sin(\log(1-\phi))+6\right]$. It is easy to verify that $\dot{z}(1)$ does not exist.
}
\end{remark}

\section{The behavior of $z$ near $0$}
\label{ssec:dotz at zero}
\setcounter{equation}{0}

For $\phi_0\in (0, 1)$ we consider the problem, see Figure \ref{f:3.1} on the left,
\begin{equation}
\label{first order problem hat}
\begin{cases}
\dot{z}(\varphi)=h(\varphi)-c-\frac{q(\varphi)}{z(\varphi)}, \ \varphi\in (0, 1),\\
z(\phi_0)=z^*(\phi_0).
\end{cases}
\end{equation}

\begin{lemma}\label{e:zsd}
Assume {\em (q)}. Fix $c>c^*$. For every $\phi_0\in (0, 1)$ there is a unique solution $\hat z_{\phi_0}\in C[0, 1]\cap C^1\left(0, 1\right)$ to problem \eqref{first order problem hat}. We have $\hat z_{\phi_0}(0) =0$, and also
\begin{equation}\label{e:fake}
\hat{z}_{\phi_0} < z^*\ \hbox{ in }\ (\phi_0, 1]\quad \hbox{ and } \quad \hat{z}_{\phi_0}\ge z_\beta\ \hbox{ in }\ (0, 1],
\end{equation}
where $z_\beta$ is the solution to \eqref{first order problem} with $z_\beta(1)=\beta$.
If $0<\phi_1<\phi_0$ then $\hat{z}_{\phi_1}<\hat{z}_{\phi_0}$ in $(0,1]$.
\end{lemma}
\begin{proof}
The existence and uniqueness of solutions is proved by {\em Step (i)} in the proof of Proposition \ref{prop: first order ii}. Inequality $\eqref{e:fake}_1$  follows from the arguments contained in {\em Step (i)} of the proof of Proposition \ref{prop: first order ii}, while $\eqref{e:fake}_2$ is obvious.
\par
If $0<\phi_1<\phi_0$ then $\hat{z}_{\phi_1}(\phi_0)< \hat{z}_{\phi_0}(\phi_0)$, because $\hat{z}_{\phi_1} < z^*$ in $(\phi_1,1]$ and $\phi_0 \in (\phi_1, 1]$. The monotony follows by the uniqueness of solutions to the Cauchy problem associated to $\eqref{first order problem}_1$. The regularity of $\hat{z}_{\phi_0}$ follows from both $\eqref{first order problem hat}_1$ and Lemma \ref{lem:zlimit}; directly from $\eqref{e:fake}_2$, we deduce $\hat{z}_{\phi_0}(0)=0$.
\end{proof}

\begin{figure}[htb]
\begin{center}
\begin{tikzpicture}[>=stealth, scale=0.7]

\draw[->] (0,0) --  (6,0) node[below]{$\phi$} coordinate (x axis);
\draw[->] (0,0) -- (0,0.5) node[right]{$z$} coordinate (y axis);
\draw (0,0) -- (0,-5);
\draw (5,0) -- (5,-5);
\draw[thick] (0,0) .. controls (2,-1.5) and (3.5,-1.5) .. node[very near end, below=-3]{$z^*$} (5,0) node[above]{\footnotesize{$1$}}; 
\draw[thick] (3.5,-1) .. controls (3.7,-1.1) and (3.5,-1.2) .. node[very near end, below=-3]{$\hat z_{\phi_0}$}  (5,-1.5) ; 
\draw[thick,dotted] (0,0) .. controls (1,-0.2) and (3,-0.5) .. (3.5,-1); 
\draw[dotted] (3.5,0) node[above]{\footnotesize{$\phi_0$}}-- (3.5,-1);

\draw[thick] (1.6,-0.9) .. controls (2.3,-2) and (3.5,-2.3) .. node[very near end, below=-3]{$\hat z_{\phi_1}$}  (5,-2.5) ; 
\draw[thick,dotted] (0,0) .. controls (1,-0.2) and (1.4,-0.5) .. (1.6,-0.9); 
\draw[dotted] (1.6,0) node[above]{\footnotesize{$\phi_1$}}-- (1.6,-0.9); 

\draw[thick] (0,0) .. controls (1,-3) and (4,-3.3) .. node[near end, below=-3]{$\hat z$}  (5,-3.5) node[right]{\footnotesize{$\hat \beta$}}; 

\draw[thick] (0,0) .. controls (1,-4) and (4,-4.3) .. node[near end, below=-3]{$z_\beta$}  (5,-4.5) node[right]{\footnotesize{$\beta$}}; 

\begin{scope}[xshift=10cm]
\draw[->] (0,0) --  (6,0) node[below]{$\phi$} coordinate (x axis);
\draw[->] (0,0) -- (0,0.5) node[right]{$z$} coordinate (y axis);
\draw (0,0) -- (0,-4.2);
\draw (5,0) -- (5,-4.2);
\draw[thick] (0,0) .. controls (2,-0.7) and (4,0) .. node[midway, below=5,right=3]{$z_b$} (5,-1) node[right]{$b$}; 
\draw (0,0) -- (2,-0.8) node[right]{\footnotesize{$s_+$}};
\draw (0,0) -- (2,-1.2) node[right]{\footnotesize{$s_+^*$}};
\draw (0,0) -- (2,-2.2) node[right]{\footnotesize{$s_-^*$}};
\draw (0,0) -- (2,-4) node[left]{\footnotesize{$s_-$}};
\draw[thick, dashed] (0,0) .. controls (2,-2.2) and (4,-2.2) .. node[midway, below=3, right=3]{$z^*$} (5,0); 
\draw[thick] (0,0) .. controls (2,-3.8) and (4,-2.4) .. node[near end, above]{$z_{\hat \beta}$} (5,-3.4) node[right]{$\hat \beta$}; 
\draw[thick] (0,0) .. controls (2,-4) and (4,-3) .. node[near end, below]{$z_{\beta}$} (5,-4) node[right]{$\beta(c)$}; 
\draw (5,0)  node[above]{\footnotesize{$1$}};

\end{scope}
\end{tikzpicture}
\end{center}

\caption{\label{f:3.1}{Left: the functions $\hat{z}_{\phi_0}$, $\hat z$ and $z_\beta$ in Lemma \ref{e:zsd}. Right: an illustration of Proposition \ref{prop:dotz at zero} for fixed $c>c^*$. Solutions are labelled according to their right-hand limit; $s_\pm$ denote the slope of the tangent of $z$ at $0$. The dashed curve is the plot of $z^*$.}}
\end{figure}
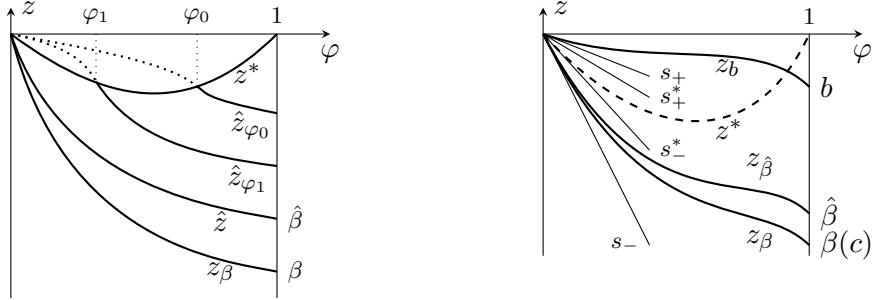


For every $c>c^*$, by the monotonicity of $\{\hat{z}_{\phi_0}\}_{\phi_0}$ and $\eqref{e:fake}_2$, Lemma \ref{lem:convergence} implies that there exists $\hat{z}\in C^0[0,1]\cap C^1\left(0, 1\right)$ which solves \eqref{first order problem0} such that
\begin{equation}
\label{hatz}
\hat{z}(\phi)=\lim_{\phi_0\to 0^+}\hat{z}_{\phi_0}(\phi), \ \phi \in \left[0,1\right].
\end{equation}
Such a $\hat{z}$ satisfies $z_\beta\leq \hat{z} \le z^*$ in $(0,1)$ by \eqref{e:fake} and then $\eqref{first order problem}$. Define $\hat{\beta}\in [\beta, 0)$ by
\begin{equation}
\label{e:hatbeta}
\hat{\beta}:=\hat{z}(1).
\end{equation}
In the following result we assume again \eqref{e:D1g1 at zero}. We shall prove in Remark \ref{r:ex_Diego} that such a condition is necessary for the existence of $\dot z(0)$. From \eqref{estimates on c*} and \eqref{e:D1g1 at zero} we deduce $\left(h(0)-c\right)^2 - 4\dot q(0) \geq 0$ for any $c\geq c^*$; we can then denote
$$
s_\pm(c):= \frac{h(0)-c}{2} \pm \frac{\sqrt{\left(h(0)-c\right)^2 - 4 \dot q(0)}}{2},\quad \hbox{ for } c\ge c^*.
$$
The next proposition generalizes \cite[Proposition 5.2]{CM-DPDE} to the case of a more generic $q$, and, more deeply, to the case $z(1)<0$. It is worth noting that this latter case reveals  the behavior detected by \eqref{e:dotz 2}, and shown in Figure \ref{f:3.1} on the right, which was not contained in \cite{CM-DPDE}.

\begin{proposition}
\label{prop:dotz at zero}
Assume {\em (q)} and \eqref{e:D1g1 at zero}.
If $c\geq c^*$ and $z$ is a solution of \eqref{first order problem}, then, $\dot{z}(0)$ exists.
Moreover, it holds that
\begin{equation}
\label{derivative of z at zero}
\dot{z}(0) =
\left\{
\begin{array}{ll}
s_+(c)\ &\mbox{ if  $c>c^*$ and  $z(1) > \hat{\beta}$},
\\[2mm]
s_-(c^*) \ &\mbox{ if $c=c^*$},
\end{array}
\right.
\end{equation}
and, if $c^*>h(0)$,
\begin{equation}
\label{e:dotz 2}
\dot{z}(0)=
s_-(c)\ \mbox{ if } \ c>c^* \ \mbox{ and } \ z(1) \in [\beta, \hat{\beta}].
\end{equation}
\end{proposition}

%

\begin{proof}
Arguing as in the proof of \cite[Proposition 5.2]{CM-DPDE}, we deduce that $\dot z(0)$ exists for $c\ge c^*$ and is one of the root of the equation
%
$\gamma^2-\left(h(0)-c\right)\gamma + \dot q(0)=0$.
%
Then $\dot{z}(0) \in \left\{ s_-(c), s_+(c)\right\}$ for every $c\geq c^*$. Straightforward computations give
\begin{equation}\label{e:ssss}
s_-(c)<s_-(c^*)\leq s_+(c^*) \leq  s_+(c)\leq 0\ \mbox{ for any } \ c>c^*
\end{equation}
and $h(0)-c \leq s_-(c)$, for any $c\geq c^*$. We denote $s_\pm^*:=s_\pm(c^*)$.
\par
\smallskip
Take $c>c^*$. Let $\hat{z}_{\phi_0}$ and $\hat{z}$ be defined as in the beginning of Section \ref{ssec:dotz at zero}, see Figure \ref{f:3.1} on the left. If $z(1)>\hat{\beta}$ then $z(1)>\hat{z}_{\phi_1}(1)$, for some $\phi_1\in (0,1)$, because of \eqref{hatz}. Thus, $z>\hat{z}_{\phi_1}$ in $(0,1]$. We observed in \eqref{e:asd} that $\hat{z}_{\phi_1}>z^*$ in $(0,\phi_1)$. Thus, $z>z^*$ in $(0,\phi_1)$ and hence $\dot{z}(0) \geq \dot{z}^*(0)$. Since $s_-(c)<s_-^*\leq 0$ by \eqref{e:ssss}, we deduce $\dot{z}(0)=s_+(c)$. This proves $\eqref{derivative of z at zero}_1$.
 \smallskip
\par
Now, we prove $\eqref{derivative of z at zero}_2$. If $z=z^*$, then $\eqref{derivative of z at zero}_2$ was obtained in \cite[Proposition 5.2]{CM-DPDE} under some specific assumptions on $q$. Since the relevant ones were \eqref{e:q} and \eqref{e:D1g1 at zero}, we deduce that $\eqref{derivative of z at zero}_2$ occurs also in the current case. If $z=y^*$ is a solution of \eqref{first order problem}, different from $z^*$ (such a $y^*$ can exist, as we proved in Remark \ref{rem:c*=0}, since \eqref{e:c=c^*2e} does not necessarily follow), then $y^* < z^*$ in $(0,1]$ by Proposition \ref{prop: first order}. Since $\dot{y^*}(0)\in \left\{s_{-}^*, s_{+}^*\right\}$ and $\dot{z^*}(0)=s_{-}^*$ then we have $\dot{y^*}(0)=s_{-}^*$. Hence, $\eqref{derivative of z at zero}_2$ holds.
\par
It remains to prove \eqref{e:dotz 2} under the additional condition $h(0)-c^*<0$. By $\beta \leq z(1) \leq \hat{\beta}$ we have $z\leq \hat{z}$ and hence $z < z^*$, which implies $\dot{z}(0) \leq \dot{z}^*(0)$. Since, under the additional condition $h(0)-c^*<0$, we have $s_-^*<s_+^*$ and since we proved that $\dot{z}^*(0)=s_-^*$, we conclude that $\dot{z}(0)=s_-(c)$, which is \eqref{e:dotz 2}. This concludes the proof.
\end{proof}

\begin{remark}
\label{r:ex_Diego}
{
\rm
Now, we prove that \eqref{e:D1g1 at zero} is necessary for the existence of $\dot{z}(0)$. For $\phi \in [0,1]$ define
$q(\phi)=\phi(1-\phi)^4\left(2+\sin\left(\log \phi\right)\right)\left(3-\cos\left(\log\phi\right)-\sin\left(\log\phi\right)\right)$. The function $q$ satisfies (q), while $\dot{q}(0)$ does not exist, since
$\liminf_{\phi \to 0^+}q(\phi)/\phi < \limsup_{\phi \to 0^+}q
(\phi)/\phi$.
Direct computations show that the function $z(\phi)=-\left(2+\sin\left(\log\phi\right)\right)\left(1-\phi\right)^2\phi$ solves \eqref{first order problem} with $c=0$ and $h(\phi)=2\left(2+\sin\left(\log\phi\right)\right)\left(1-\phi\right)\phi-5\left(1-\phi\right)^2$. Clearly, $\dot{z}(0)$ does not exists.
}
\end{remark}

We now show that, under the assumptions of Proposition \ref{prop: first order 3}, the threshold $\hat\beta(c)$ defined in \eqref{e:hatbeta} and occurring in Proposition \ref{prop:dotz at zero} coincides with the threshold $\beta(c)$ introduced in Proposition \ref{prop: first order ii}. It is an open problem whether the two thresholds differ without assuming \eqref{e:D1g1 at zero} and \eqref{e:c=c^*2e}.

\begin{proposition}\label{p:beta=hatbeta}
Assume {\em (q)}, \eqref{e:D1g1 at zero}, \eqref{e:c=c^*2e} and $c>c^*$. Then $\beta(c)=\hat{\beta}(c)$.
\end{proposition}

\begin{proof}
Consider $\eps>0$ and let $z_\eps$ be the solution of
\begin{equation*}
\begin{cases}
\dot{z_\eps}(\phi)=h(\phi)-c-\frac{q(\phi)}{z_\eps(\phi)},  \phi>0,\\
z_\eps(0)=-\eps<0.
\end{cases}
\end{equation*}
Lemma \ref{lem:cm-dpde} {\em (1.b)} implies that $z_\eps$ exists and it is defined in its maximal-existence interval $[0,\delta]$, for some $0< \delta\le 1$.  By the uniqueness of solutions of the Cauchy problem associated to $\eqref{first order problem0}_1$, we have necessarily $z_\eps<z_\beta$ in $[0,\delta]$, where $z_\beta$ was defined in the statement of Lemma \ref{e:zsd}. Since $z_\beta(\delta)<0$ then $\delta=1$.

We claim that $z_\eps$ converges for $\eps \to 0^+$ to both $\hat{z}$ and $z_\beta$, where $\hat{z}$ is defined in \eqref{hatz}, see Figure \ref{f:3.1} on the left. From the uniqueness of the limit, it follows that $\hat{z}$ and $z_\beta$ coincides and hence that $\beta=\hat{\beta}$.
To prove the claim, consider
\[
\eta_\eps(\phi):=\hat{z}(\phi)-z_\eps(\phi), \ \phi \in [0,1].
\]
Since $\hat{z}\ge z_\beta>z_\eps$ in $[0,1]$, then $\eta_\eps>0$ in $[0,1]$. Moreover, $\eta_\eps(0)=\eps$. We have
\[
\dot{\eta_\eps}(\phi)=\frac{q(\phi)}{z_\eps(\phi)\hat{z}(\phi)}\eta_\eps(\phi), \ \phi \in (0,1).
\]
Thus,
\[
\frac{\dot{\eta_\eps}(\phi)}{\eta_\eps(\phi)}=\frac{q(\phi)}{z_\eps(\phi)\hat{z}(\phi)}, \ \phi \in(0,1)
\]
and hence, for any $0< \tau < \phi$,
\begin{equation}
\label{e:etaeps}
\log\left(\eta_\eps(\phi)\right) -\log\left(\eta_\eps(\tau)\right)=\int_{\tau}^{\phi}\frac{q(s)}{z_\eps(s)\hat{z}(s)}\,ds \le \int_{\tau}^{1}\frac{q(s)}{z_\beta(s)\hat{z}(s)}\,ds.
\end{equation}
Notice, from $\eqref{e:c=c^*2e}_2$ it follows that we can apply \eqref{e:dotz 2} with $\dot q(0)=0$ (because of \eqref{e:D1g1 at zero}) and obtain $z_\beta(s)\hat{z}(s)=\left(h(0)-c\right)^2 s^2 + o(s^2)$, as $s\to 0^+$. Hence,  from $\eqref{e:c=c^*2e}_1$,
\[
\sup_{\tau >0} \int_{\tau}^{1} \frac{q(s)}{z_\beta(s)\hat{z}(s)}\,ds=:C<+\infty.
\]
From \eqref{e:etaeps}, by taking the limit as $\tau\to 0^+$ we deduce $\eta_\eps(\phi) \le \eps e^C$, $\phi \in [0,1)$, and then
\begin{equation}
\label{e:etaeps2}
\lim_{\eps \to 0^+}
z_\eps(\phi)=\hat{z}(\phi), \ \phi \in [0,1).
\end{equation}
We now apply Lemma \ref{lem:convergence} to deduce that $z_\eps$ converges (uniformly on $[0,1]$) to a solution $\bar{z}$ of $\eqref{first order problem0}_1$ in $(0,1)$ such that $\bar{z}<0$ in $(0,1)$ and $\bar{z}(0)=0$. Since $z_\eps < z_\beta$ and $z_\beta$ lies below every solution of \eqref{first order problem0}, by the very definition of $z_\beta$, we conclude that $\bar{z}$ coincides with $z_\beta$, that is
$\lim_{\eps \to 0^+} z_\eps(\phi)=z_\beta(\phi)$, $\phi \in [0,1]$.
From this formula and \eqref{e:etaeps2} we clearly have $z_\beta=\hat{z}$.\end{proof}

\section{Strongly non-unique strict semi-wavefronts}
\label{s:existence_0alpha}
\setcounter{equation}{0}

We now apply the previous results to study semi-wavefronts of Equation \eqref{e:E} when $D$ and $g$ satisfy (D1), (g0) and \eqref{Dg}; in particular, we prove Theorem \ref{th:swf to zero} and Corollary \ref{cor:swf qualitative behavior}. Indeed, all the results obtained in Sections \ref{s:sing2}--\ref{ssec:dotz at zero} apply when we set
\begin{equation}
\label{e:q=Dg}	
q:=Dg,
\end{equation}
since such $q$ fulfills (q). Throughout this section, by $c^*$ we always intend the threshold given by Proposition \ref{prop: first order} for $q$ as in \eqref{e:q=Dg}, for which it holds \eqref{estimates on c* new}, as observed in Remark \ref{rem:MP}.

\begin{lemma}
\label{lem:D1fracz}
Assume {\em (D1)}, {\em (g0)} and \eqref{Dg}. Consider $c\geq c^*$ and let $z$ be the solution of \eqref{first order problem 0-0} when \eqref{e:q=Dg} occurs. Then, it holds that
\begin{equation}
\label{e:D1fracz}
\lim_{\phi \to 1^-} \frac{D(\phi)}{z(\phi)}=
\left\{
\begin{array}{ll}
\frac{h(1)-c-\sqrt{\left(h(1)-c\right)^2-4\dot{D}(1)g(1)}}{2g(1)} \ &\mbox{ if } \ \dot{D}(1)<0,\\[10pt]
\min\left\{0, \frac{h(1)-c}{g(1)}\right\} \ &\mbox{ if } \ \dot{D}(1)=0.
\end{array}
\right.
\end{equation}
\end{lemma}

\begin{proof}
First, observe that Proposition \ref{prop:derivative of z at alpha} applies to the current case.

If either $\dot{D}(1)<0$ or $\dot{D}(1)=0$ and $c<h(1)$, then $\dot{z}(1)>0$, by \eqref{derivative of z at alpha}, because $\dot{q}(1)=\dot{D}(1)g(1)$. As a consequence, we have
\begin{equation*}
\lim_{\phi \to 1^-}
\frac{D(\phi)}{z(\phi)}=\lim_{\phi\to 1^-}\frac{\frac{D(\phi)}{\phi - 1}}{\frac{z(\phi)}{\phi - 1}}=\frac{\dot{D}(1)}{\dot{z}(1)},
\end{equation*}
which, together with \eqref{derivative of z at alpha}, implies both $\eqref{e:D1fracz}_1$ and the first half of $\eqref{e:D1fracz}_2$.

\smallskip

If $\dot{D}(1)=0$ and $c\geq h(1)$, we need a refined argument based on strict upper- and lower-solutions of $\eqref{first order problem}_1$. We split the proof in two subcases.

\smallskip

\noindent {\em (i)} Assume first $\dot{D}(1)=0$ and $c>h(1)$. Fix $\eps >0$ and define $\omega=\omega(\varphi)$ by
\begin{equation}\label{e:ur_omega}
\omega(\varphi):=- \frac{g(1)}{c-h(1)+\varepsilon g(1)}D(\varphi), \quad \hbox{ for } \varphi\in(0,1).
\end{equation}
First, we observe that $\omega <0$ in $(0,1)$. Moreover, we get
\begin{equation*}
\dot{\omega}(\varphi)=-\frac{g(1)}{c-h(1) +\varepsilon g(1)}\dot{D}(\varphi),
\end{equation*}
which in turn implies $\dot{\omega}(1)=0$, since $\dot{D}(1)=0$. Now, if we compute the right-hand side of  $\eqref{first order problem}_1$ applied to ${\omega}$, we obtain
\begin{equation*}
h(\varphi) - c -\frac{D(\varphi)g(\varphi)}{\omega(\varphi)} = h(\varphi) - c +\frac{g(\varphi) \left[ c- h(1) +\varepsilon g(1)\right]}{g(1)}, \quad \mbox{ for } \ \varphi \in (0,1),
\end{equation*}
which tends to $\varepsilon g(1) >0$ as $\varphi \to 1^-$. Hence, there exists $\sigma \in (0,1)$ such that
\begin{equation}
\label{e:lower solution}
\dot{\omega}(\varphi) < h(\varphi) -c -\frac{D(\varphi)g(\varphi)}{\omega(\varphi)}, \quad \varphi \in [\sigma, 1),
\end{equation}
that is, $\omega$ is a (strict) lower-solution of $\eqref{first order problem}_1$ in $[\sigma, 1)$.
\par
Since $\dot{z}(1)=0$, we can take a sequence $\left\{\varphi_n\right\}_n\subset (\sigma,1)$, with $\phi_n\to 1$ as $n\to\infty$, such that $\dot{z}(\varphi_n)\to 0$ as follows. Let  $\{\sigma_n\}_n  \subset (\sigma, 1)$ be such that $\sigma_n \to 1$. For any $n\in \N$, the Mean Value Theorem implies that there exists $\varphi_n \in (\sigma_n, 1)$ for which it holds $\dot{z}(\varphi_n)= \frac{z(\sigma_n)}{\sigma_n -1}$. Since the sequence in the right-hand side of this last identity tends to $\dot{z}(1)=0$, as $n\to \infty$, we obtained the desired $\{\phi_n\}_n$.
With this in mind, from $\eqref{first order problem}_1$, we obtain
\begin{equation}
\label{sequence}
\lim_{n\to \infty} \frac{D(\varphi_n)g(\varphi_n)}{z(\varphi_n)} = h(1) - c,
\end{equation}
and then
\begin{equation*}
\lim_{n\to \infty} \frac{\omega(\varphi_n)}{z(\varphi_n)}= \frac{c-h(1)}{c-h(1) +\varepsilon g(1)}=1- \frac{\varepsilon g(1)}{c-h(1)+\varepsilon g(1)}<1.
\end{equation*}
Hence, there exists $\overline n$ such that $\omega(\phi_n)>z(\phi_n)$ for $n\ge \overline n$. Without loss of generality we assume that $\overline n=1$. We claim that
\begin{equation}\label{e:phizz}
\omega(\phi)>z(\phi), \quad \hbox{ for } \phi \in (\phi_1, 1).
\end{equation}

We reason  by contradiction, see Figure \ref{f:fp}. Suppose that there exists $\tilde\phi \in(\phi_1, 1)$ such that $\omega(\tilde \phi)\leq z(\tilde\phi)$. There exists $n \in\N$ for which $\tilde\phi \in (\phi_n, \phi_{n+1})$. Since
$\omega(\phi_n)>z(\phi_n)$ and $\omega(\phi_{n+1})>z(\phi_{n+1})$,
the existence of such a $\tilde\phi$ implies that the function $\left(\omega-z\right)$ in $(\phi_n, \phi_{n+1})$ admits a non-positive minimum at $\tilde\phi_2\in (\phi_n, \phi_{n+1})$, that is $\dot{\omega}(\tilde\phi_2)=\dot{z}(\tilde\phi_2)$ and $\omega(\tilde\phi_2)\leq z(\tilde\phi_2)$. Thus, from $\eqref{first order problem}_1$ and \eqref{e:lower solution} we have that
$$
h(\tilde\phi_2)-c-\frac{(Dg)(\tilde\phi_2)}{z(\tilde\phi_2)} = \dot{z}(\tilde\phi_2)=\dot{\omega}(\tilde\phi_2) < h(\tilde\phi_2)-c-\frac{(Dg)(\tilde\phi_2)}{\omega(\tilde\phi_2)},
$$
which in turn implies $1/z(\tilde\phi_2) > 1/\omega(\tilde\phi_2)$
because of $(Dg)(\tilde\phi_2) >0$. Hence, $z(\tilde\phi_2) < \omega(\tilde\phi_2)$ which contradicts the existence of $\tilde\phi_2$. Then \eqref{e:phizz} is proved. At last, we have
\begin{equation}
\label{estimate 1}
\frac{D(\varphi)}{z(\varphi)} > \frac{D(\varphi)}{\omega(\varphi)} = -\frac{c - h(1)}{g(1)} - \varepsilon, \quad  \varphi \in (\phi_1,1).
\end{equation}

\begin{figure}[htb]
\begin{center}
\begin{tikzpicture}[>=stealth, scale=0.7]

\draw[->] (0,0) --  (6,0) node[below]{$\phi$} coordinate (x axis);
\draw[->] (0,0) -- (0,0.5) node[left]{$z$} coordinate (y axis);
\draw (0,0) -- (0,-2.5);
\draw[thick] (1.5,-2.5) .. controls (2.5,0) and (3.5,0) .. node[near end, above]{$z$} (4.5,-2) ; 
\draw[thick] (1.5,-2) .. controls (2.5,-1.5) and (3.5,-1.5) .. node[near start, below]{$\omega$} (4.5,-1.5) ; 

\draw[dotted] (1.5,-2.5) -- (1.5,0) node[above]{\footnotesize{$\phi_n$}};
\draw[dotted] (4.5,-2) -- (4.5,0) node[above]{\footnotesize{$\phi_{n+1}$}};

\draw (5.5,0)  node[above]{\footnotesize{$1$}};
\draw (5.5,-0.1) -- (5.5,0.1);
\draw (0.5,0)  node[above]{\footnotesize{$\phi_1$}};
\draw (0.5,-0.1) -- (0.5,0.1);
\draw[dotted] (3.1,-1.5) -- (3.1,0) node[above]{\footnotesize{$\tilde\phi_2$}};
\draw[dotted] (2.3,-1.7) -- (2.3,0) node[above]{\footnotesize{$\tilde\phi$}};

\end{tikzpicture}
\end{center}

\caption{\label{f:fp}{A detail of the plots of functions $\omega$ and $z$ in case {\em (i)}.}}
\end{figure}
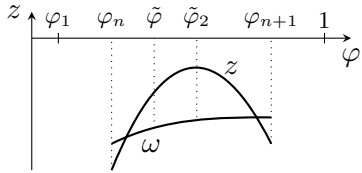

Analogously, for $\varepsilon>0$ small enough to satisfy $c>h(1) + \varepsilon g(1)$, we define $\eta=\eta(\varphi)$ by
\begin{equation*}
\eta(\varphi):=- \frac{g(1)}{c-h(1)-\varepsilon g(1)}D(\varphi), \quad \varphi\in(0,1),
\end{equation*}
By arguing as above when we considered $\omega$ in \eqref{e:ur_omega}, we deduce that
%
$\eta$ is a (strict) upper-solution of $\eqref{first order problem}_1$ in $[\sigma_2, 1)$ for some $\sigma_2 \in (0,1)$. Proceeding as we did to obtain \eqref{estimate 1}, we now get
%
$\eta(\phi)<z(\phi)$ for $\phi \in (\phi_1, 1)$, for some $\phi_1>\sigma_2$. Thus,
\begin{equation}
\label{estimate 2}
\frac{D(\varphi)}{z(\varphi)} < \frac{D(\varphi)}{\eta(\varphi)} = -\frac{c - h(1)}{g(1)} + \varepsilon, \quad  \varphi \in (\phi_1,1).
\end{equation}
Finally, putting together \eqref{estimate 1} and \eqref{estimate 2}, since $\varepsilon>0$ is arbitrary, we deduce
\begin{equation}
\label{D1 frac z limit}
\lim_{\varphi \to 1^-} \frac{D(\varphi)}{z(\varphi)}= \frac{h(1) - c }{g(1)}.
\end{equation}
Thus, we proved $\eqref{e:D1fracz}_2$ with $c>h(1)$.

\smallskip

\noindent {\em (ii)} Now, we consider the case $\dot{D}(1)=0$ and $c = h(1)$. Fix $\eps >0$. Set
\begin{equation}\label{e:urn_omega}
\omega(\varphi):=- \frac{D(\varphi)}{\eps}, \ \phi \in (0,1),
\end{equation}
which coincides with \eqref{e:ur_omega} in the current case. By proceeding exactly as in the case {\em (ii)}, we obtain \eqref{e:ur_omega} for $\omega$ defined as in \eqref{e:urn_omega}, namely $0>\omega(\phi) >z(\phi)$, for $\phi \in (\phi_1,1)$, for some $\phi_1\in (0,1)$. This implies, as in \eqref{estimate 1},
\begin{equation}\label{e:sif}
0> \frac{D(\phi)}{z(\phi)} > \frac{ D(\phi)}{\omega(\phi)} = - \eps, \quad \phi \in (\phi_1, 1).
\end{equation}
Then \eqref{e:sif} implies $D(\phi)/z(\phi)  \to 0^-$ as $\phi \to 1^-$,
which is $\eqref{e:D1fracz}_2$ in the case $c=h(1)$.
\end{proof}

\begin{remark}
\label{rem:z c1}
{\rm
Let $c\geq c^*$ and $z$ be any solution of \eqref{first order problem}. We infer that $z\in C^1(0,1]$. In fact, if $z(1)=b<0$, in the proof of case {\em (i)} of Proposition \ref{prop:derivative of z at alpha} we already checked that this is true, since $\lim_{\phi\to 1^-}\dot{z}(\phi)=\dot{z}(1)$. If $z(1)=0$, from \eqref{e:D1fracz} it follows that the right-hand side of $\eqref{first order problem}_1$ still has a finite limit, as $\phi \to 1^-$. As observed, this means that $z\in C^1(0,1]$.
}
\end{remark}

We now prove Theorem \ref{th:swf to zero}.
%

\begin{proof}[Proof of Theorem \ref{th:swf to zero}]
We first prove that there exists a semi-wavefront to $0$ of \eqref{e:E} {\em if } $c\geq c^*$. For $q=Dg$, consider one of the solutions $z=z(\varphi)$ of \eqref{first order problem}, provided by Propositions \ref{prop: first order} and \ref{prop: first order ii}. Consider the Cauchy problem
\begin{equation}
\label{cauchy}
\begin{cases}
\varphi\rq{}=\frac{z(\varphi)}{D(\varphi)},\\
\varphi(0)=\frac{1}{2}.
\end{cases}
\end{equation}
The right-hand side of $\eqref{cauchy}_1$ is of class $C^1$ in a neighborhood of $\frac1{2}$, and then there exists a unique solution $\varphi$ in its maximal-existence interval $(a,\xi_0)$, for $-\infty\le a<\xi_0\le\infty$. Since $z(\phi)/D(\phi)<0$ for $\phi\in(0,1)$, we deduce that $\phi$ is decreasing and then
$\lim_{\xi \to a^+} \varphi(\xi)= 1$, $ \lim_{\xi \to \xi_0^-} \varphi(\xi)=0$.
By $\eqref{cauchy}_1$, the profile $\varphi$ satisfies \eqref{e:ODE} in $(a,\xi_0)$. We show that, if $\xi_0\in\R$, we can extend $\varphi$ and obtain a solution of \eqref{e:ODE}, in the sense of Definition \ref{d:tws}, defined in the half-line $(a,+\infty)$.

Assume $\xi_0\in\R$ and set $\varphi(\xi)=0$, for any $\xi\geq \xi_0$.  The new function  (which without any ambiguity we still call $\varphi$) is clearly of class $C^0(a,+\infty)\cap C^2\left(\left(a,+\infty\right) \setminus\{\xi_0\}\right)$ and is a classical solution of \eqref{e:ODE} in $(a,+\infty)\setminus\{\xi_0\}$. Moreover, observe that, as a consequence of both the fact that $z$ satisfies $\eqref{first order problem}_3$, and $\eqref{cauchy}_1$, we have
\begin{equation}
\label{asymptotic behavior to b-}
\lim_{\xi \to \xi_0^-} D\left(\varphi (\xi)\right) \varphi\rq{}(\xi)=0.
\end{equation}
This implies that $D(\varphi)\varphi\rq{} \in L^1_{\rm loc}(a,+\infty)$.
\par
To show that $\varphi$ is a solution of \eqref{e:ODE} according to Definition \ref{d:tws}, it remains to prove \eqref{e:def-tw}. For this purpose, consider $\psi \in C_0^\infty(a,+\infty)$, and let $a < \xi_1 < \xi_2 < \infty$ be such that $\psi(\xi)=0$, for any $\xi \geq \xi_2$ or $\xi \leq \xi_1$.  Our goal is then to prove the following:
\begin{equation}
\label{distributional equation}
\int_{\xi_1}^{\xi_2} \left(D\left(\varphi\right)\varphi\rq{}-f\left(\varphi\right) + c \varphi\right) \psi\rq{} - g(\varphi)\psi\,d\xi=0.
\end{equation}
Identity \eqref{distributional equation} is obvious if $\xi_2<\xi_0$, since $\varphi$ solves \eqref{e:ODE} in $(a,\xi_0)$.
Assume $\xi_2 \geq \xi_0$. In the interval $(\xi_0,\xi_2)$ we have $\varphi=0$, and since $g(0)=f(0)=0$ we deduce
\begin{equation}
\label{integral 1}
\int_{\xi_0}^{\xi_2}  \left(D\left(\varphi\right)\varphi\rq{}-f\left(\varphi\right) + c \varphi\right) \psi\rq{} - g(\varphi)\psi\,d\xi =  0.
\end{equation}
In the interval $(\xi_1,\xi_0)$ we have, by \eqref{asymptotic behavior to b-},
\begin{align}
&\int_{\xi_1}^{\xi_0} \left(D\left(\varphi\right)\varphi\rq{}-f\left(\varphi\right) + c \varphi\right) \psi\rq{} - g(\varphi)\psi\,d\xi
\nonumber\\
=&\lim_{\eps\to0^+}\int_{\xi_1}^{\xi_0-\eps} \left(D\left(\varphi\right)\varphi\rq{}-f\left(\varphi\right) + c \varphi\right) \psi\rq{} - g(\varphi)\psi\,d\xi
\nonumber\\
=&\lim_{\eps\to0^+}\left(\left(D\left(\phi\right)\phi\rq{}-f(\phi)+c\phi\right)\psi\right)(\xi_0-\eps)
=0.
\label{integral 2}
\end{align}
Thus, identities \eqref{integral 1} and \eqref{integral 2} imply \eqref{distributional equation}.

\smallskip

At last, we claim that $a\in \R$, i.e., that $\varphi$ is {\em strict}. For this, it is sufficient to prove
\begin{equation}\label{e:phipneg}
\lim_{\xi \to a^+} \varphi\rq{}(\xi)<0.
\end{equation}
We stress that the case $\lim_{\xi\to a^+}\phi\rq{}(\xi)\to -\infty$, for short $\phi\rq{}(a^+)=-\infty$, is included in \eqref{e:phipneg}. To prove \eqref{e:phipneg}, we notice that, from \eqref{cauchy},
\begin{equation*}\label{e:phiplim}
\lim_{\xi \to a^+} \varphi\rq{}(\xi)= \lim_{\varphi \to 1^-} \frac{z(\varphi)}{D(\varphi)}.
\end{equation*}
Thus, \eqref{e:phipneg} easily follows from either a direct check, in the case $z(1)<0$, or the application of Lemma \ref{lem:D1fracz}, in the case $z(1)=0$. This concludes the first part of the proof.

\smallskip

Conversely, we prove that if there exists a semi-wavefront $\varphi$ to $0$ defined in $(a,+\infty)$, then $c\geq c^*$. Let $\bar{b}$ be defined by
\begin{equation}
\label{e:bar b}
\bar{b}:= \sup \left\{\xi >a: \phi(\xi)>0\right\} \in (a, +\infty].
\end{equation}
We have $0<\phi < 1$ in $\left(a,\bar{b}\right)$ and so $\phi$ is a classical solution of \eqref{e:ODE} in $\left(a,\bar{b}\right)$.
 We claim that
\begin{equation}
\label{asymptotic behaviour}
\lim_{\xi \to \bar{b}^-} D\left(\varphi(\xi)\right) \varphi\rq{}(\xi)=0.
\end{equation}
Suppose $\bar{b}\in \R$. Take $\xi_1>a$ and $\xi_2 >\bar{b}$. By choosing, in Definition \ref{d:tws}, $\psi \in C_0^\infty(a,+\infty)$ with support in $(\xi_1,\xi_2)$ such that $\psi(\bar{b}) \neq 0$, \eqref{e:def-tw} reads as (passing to the limit in the integral as in \eqref{integral 2})
\begin{multline*}
0=\int_{\xi_1}^{\xi_2}\left(D\left(\phi\right)\phi\rq{}+ c \phi - f\left(\phi\right)\right)\psi\rq{}-g\left(\phi\right)\psi\, d\xi=\\
\int_{\xi_1}^{\bar{b}} \left(D\left(\phi\right)\phi\rq{} + c \phi - f\left(\phi\right)\right)\psi\rq{}-g\left(\phi\right)\psi\, d\xi = \left(D(\phi)\phi\rq{}\right)(\bar{b}^-)\psi(\bar{b}).
\end{multline*}
Then we got \eqref{asymptotic behaviour} in this case. If $\bar{b}=+\infty$, by integrating \eqref{e:ODE} in $[\eta, \xi]\subset (\bar{a}, +\infty)$, we have
\begin{equation}
\label{e:Nik}
D\left(\varphi(\xi)\right)\varphi\rq{}(\xi) = D\left(\varphi(\eta)\right)\varphi\rq{}(\eta) - c \left(\varphi(\xi)-\varphi(\eta)\right) + \left(f(\varphi(\xi))-f(\varphi(\eta))\right) - \int_{\eta}^{\xi} g\left(\varphi(\sigma)\right)\,d\sigma.
\end{equation}
Since the function
$$
\xi \mapsto \int_\eta^\xi g(\varphi(\sigma))\,d\sigma
$$
is increasing (because $g>0$ in $(0,1)$), then $\lim_{\xi\to \infty} D\left(\varphi(\xi)\right)\varphi\rq{}(\xi)=\ell$ for some $\ell \in [-\infty, 0]$. If $\ell <0$, then, $\varphi\rq{}(\xi)$ tends either to some negative value or to $-\infty$ as $\xi \to +\infty$. In both cases, this contradicts the boundedness of $\varphi$, and so \eqref{asymptotic behaviour} is proved.

\smallskip

We show now \eqref{strict monotonicity}. Suppose by contradiction that \eqref{strict monotonicity} does not occur, there exists $\xi_0 \in (a,\bar{b})$, with $0<\varphi(\xi_0) < 1$, such that $\varphi\rq{}(\xi_0)=0$. Then \eqref{e:ODE} implies $\varphi\rq{}\rq{}(\xi_0)= -g\left(\varphi(\xi_0)\right)/D\left(\varphi(\xi_0)\right) <0$ and hence $\xi_0$ is a local maximum point of $\varphi$. It is plain to see that, in turn, this implies that there exists $a< \xi_1 < \xi_0$ which is a local minimum point of $\varphi$. From what we said about $\xi_0$, we necessarily have $\varphi(\xi_1)=\varphi\rq{}(\xi_1)=0$.

Take $\xi \in(\xi_1, \bar{b})$. Integrating \eqref{e:ODE} in $[\xi_1,\xi]$ gives \eqref{e:Nik} with $\xi_1$ replacing $\eta$.
By passing to the limit for $\xi \to \bar{b}^-$, from \eqref{asymptotic behaviour} we obtain the contradiction $0<0$. This proves \eqref{strict monotonicity}.

From \eqref{strict monotonicity}, we can define the function $z=z(\varphi)$, for $\varphi \in (0,1)$, by
\begin{equation}
\label{e:z from phi}
z(\varphi):= D(\varphi) \varphi\rq{}\left(\xi(\varphi)\right),
\end{equation}
where $\xi=\xi(\phi)$ is the inverse function of $\phi$. Again by \eqref{strict monotonicity}, it follows also that $z<0$ in $(0,1)$. From \eqref{asymptotic behaviour}, we clearly have $z(0^+)=0$; furthermore, a direct computation shows that $z$ solves equation $\eqref{first order problem000}_1$. Thus, $z$ solves problem \eqref{first order problem000}, which is \eqref{first order problem} with $q=Dg$. At last, Proposition \ref{prop: first order 4} implies $c\geq c^*$.
\end{proof}


\begin{remark}
{\rm
The proof of Theorem \ref{th:swf to zero} provides a formula for $\phi\rq{}(a^+)$. If $z(1)<0$, then $\phi\rq{}(a^+)=-\infty$. If $z(1)=0$, Lemma \ref{lem:D1fracz} leads to
\begin{equation}
\label{e:phip at a}
\lim_{\xi \to a^+} \phi\rq{}(\xi)=
\begin{cases}
\frac{2 g(1)}{h(1)-c - \sqrt{\left(h(1)-c\right)^2 - 4 \dot{D}(1)g(1)}} \ &\mbox{ if } \ \dot{D}(1)<0,\\
\frac{g(1)}{h(1)-c} \ &\mbox{ if } \ \dot{D}(1)=0  \mbox{ and }  c>h(1),\\
-\infty \ &\mbox{ if } \ \dot{D}(1)=0  \mbox{ and } c\leq h(1).
\end{cases}
\end{equation}
}
\end{remark}

We now investigate the qualitative properties of the profiles when they reach the equilibrium $0$. The classification is complete, apart from some cases corresponding to $c^*=h(0)$, when further assumptions are needed, see Remark \ref{r:final?}. Below the existence of the $\lim_{\xi \to a^+}D\left(\phi(\xi)\right)\phi\rq{}(\xi)$ is a consequence of the definition \eqref{e:z from phi} and Lemma \ref{lem:zlimit}.

\begin{corollary}
\label{cor:swf qualitative behavior}
Under the assumptions of Theorem \ref{th:swf to zero}, let $c\ge c^*$ and $\phi$ be a strict semi-wavefront to $0$ of \eqref{e:E}, connecting $1$ to $0$, defined in its maximal-existence interval $(a,+\infty)$. Then, for $c>c^*$, there exists $\hat\beta(c)\in[\beta(c),0]$ such that the following results hold.
\begin{enumerate}[(i)]
\item  $D(0)>0$ implies that $\phi$ is classical and strictly decreasing.

\item $D(0)=0$, $c>c^*$ and
\begin{equation}
\label{e:phiclassical}
\lim_{\xi \to a^+}D\left(\phi(\xi)\right)\phi\rq{}(\xi)> \hat{\beta}(c),
\end{equation}
 imply that $\phi$ is classical; moreover, $\phi$ reaches $0$ at some $\xi_0 >a$ if
 \begin{equation}
 \label{e:g near zero}
 c>h(0)+\limsup_{\phi \to 0^+}\frac{g(\phi)}{\phi}.
 \end{equation}

 \item $D(0)=0$, $c^*>h(0)$ and
\begin{equation}
\label{e:phisharp}
\mbox{ either } \ c=c^* \ \mbox{ or } \ \lim_{\xi \to a^+}D\left(\phi(\xi)\right)\phi\rq{}(\xi)\le  \hat{\beta}(c)
\end{equation}
imply that $\phi$ is sharp at $0$ (reached at some $\xi_0>a$) with
\begin{equation}
\label{e:phiprime}
\lim_{\xi \to \xi_0^-} \phi\rq{}(\xi)=
\left\{
\begin{array}{ll}
\frac{h(0)-c}{\dot{D}(0)} <0 \ &\mbox{ if } \ \dot{D}(0)>0,\\[2mm]
-\infty \ &\mbox{ if } \ \dot{D}(0)=0.
\end{array}
\right.
\end{equation}
\end{enumerate}
\end{corollary}

Notice that $\beta$ is related to the existence of the semi-wavefronts while $\hat \beta$ deals with their smoothness (see Figure \ref{f:3profiles}). The two thresholds coincide under the assumptions of Proposition \ref{p:beta=hatbeta}.

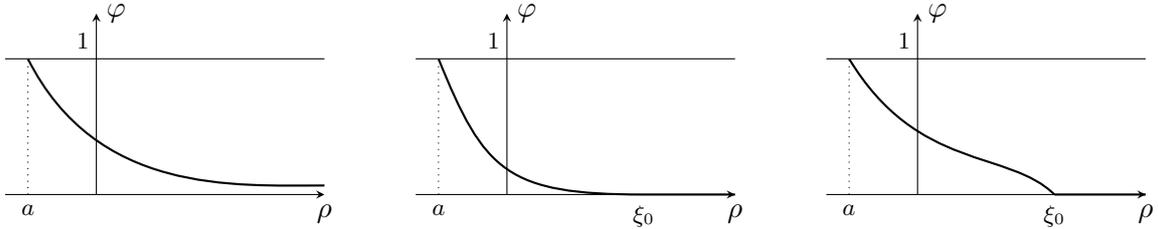
\begin{figure}[htb]
\begin{center}

\begin{tikzpicture}[>=stealth, scale=0.6]
\draw[->] (0,0) --  (5,0) node[below]{$\rho$} coordinate (x axis);
\draw       (0,0) --  (-2,0);
\draw       (-2,3) --  (5,3);
\draw[->] (0,0) -- (0,4) node[right]{$\phi$} coordinate (y axis);
\draw (0,3) node[left=5, above]{\footnotesize{$1$}};
\draw[thick] (-1.5,3) .. controls (0,0) and (3,0.2) .. (5,0.2);
\draw[dotted] (-1.5,0) node[below]{\footnotesize{$a$}} -- (-1.5,3);

\begin{scope}[xshift=9cm]
\draw[->] (0,0) --  (5,0) node[below]{$\rho$} coordinate (x axis);
\draw       (0,0) --  (-2,0);
\draw       (-2,3) --  (5,3);
\draw[->] (0,0) -- (0,4) node[right]{$\phi$} coordinate (y axis);
\draw (0,3) node[left=5, above]{\footnotesize{$1$}};
\draw[thick] (-1.5,3) .. controls (-0.5,0.5) and (0,0) .. (3,0) node[below]{\footnotesize{$\xi_0$}};
\draw[thick] (3,0) -- (5,0);
\draw[dotted] (-1.5,0) node[below]{\footnotesize{$a$}} -- (-1.5,3);
\end{scope}

\begin{scope}[xshift=18cm]
\draw[->] (0,0) --  (5,0) node[below]{$\rho$} coordinate (x axis);
\draw       (0,0) --  (-2,0);
\draw       (-2,3) --  (5,3);
\draw[->] (0,0) -- (0,4) node[right]{$\phi$} coordinate (y axis);
\draw (0,3) node[left=5, above]{\footnotesize{$1$}};
\draw[thick] (-1.5,3) .. controls (0,0.5) and (2,1) .. (3,0) node[below]{\footnotesize{$\xi_0$}};
\draw[thick] (3,0) -- (5,0);
\draw[dotted] (-1.5,0) node[below]{\footnotesize{$a$}} -- (-1.5,3);
\end{scope}

\end{tikzpicture}

\end{center}
\caption{\label{f:3profiles}{Examples of profiles occurring in Corollary \ref{cor:swf qualitative behavior}. From the left to the right, they depict, respectively, what stated in Parts {\em (i)}, {\em (ii)} and {\em (iii)}.}}
\end{figure}

\smallskip

\begin{proofof}{Corollary \ref{cor:swf qualitative behavior}}
Define $\xi_0:=\sup\left\{\xi >a: \phi(\xi)>0\right\}\in (a, +\infty]$. We assume without loss of generality that $a<0<\xi_0$ and $\phi(0)=1/2$. Let $z$ be the function defined in \eqref{e:z from phi}. Notice, $1=D(\phi)\phi\rq{}/z(\phi)$ if $\phi \in (0,1)$. Thus, for any $\xi >0$, it follows that
\begin{equation*}
\xi=\int_{0}^{\xi}\frac{D\left(\phi(s)\right)}{z\left(\phi(s)\right)}\,\phi\rq{}(s)ds=\int_{1/2}^{\phi(\xi)} \frac{D(\sigma)}{z(\sigma)}\,d\sigma=\int_{\phi(\xi)}^{1/2}\frac{D(\sigma)}{-z(\sigma)}\,d\sigma.
\end{equation*}
Therefore, $\xi_0\in\R$ if and only if it holds that
\begin{equation}
\label{e:integral control}
\int_{0}^{1/2}\frac{D(\sigma)}{-z(\sigma)}\,d\sigma:=\lim_{\phi\to 0^+}\int_{\phi}^{1/2}\frac{D(\sigma)}{-z(\sigma)}\,d\sigma < + \infty.
\end{equation}
For $c>c^*$, let $\hat{\beta}(c)$ be given by \eqref{e:hatbeta}.

We prove {\em (i)}.
By Proposition \ref{prop:dotz at zero} we know that $\dot z(0)$ exists and it is finite; since
$D(0)>0$ we deduce that \eqref{e:integral control} does not hold. Then, $\xi_0 =+\infty$ and so $\phi$ is strictly decreasing. This, and the fact that $\phi$ is of class $C^2$ when $\phi \in (0,1)$, imply $\phi\in C^2(a,+\infty)$, hence $\phi$ is classical. Part {\em (i)} is hence showed.

\smallskip

Assume $D(0)=0$. In this case, Formula \eqref{e:D1g1 at zero} holds with $\dot q(0)=0$ and $\dot{z}(0)$ exists by Proposition \ref{prop:dotz at zero}.

We show {\em (ii)}. Since \eqref{e:phiclassical} holds then \eqref{derivative of z at zero} reads as $\dot{z}(0)=0$. We treat separately the cases $\dot{D}(0)>0$ or $\dot{D}(0)=0$. Suppose that $\dot{D}(0)>0$. Therefore,
\begin{equation}
\label{e:limitphi2}
\lim_{\xi \to \xi_0^-} \phi\rq{}(\xi)=\frac{\dot{z}(0)}{\dot{D}(0)}= 0
\end{equation}
and hence $\phi$ (not necessarily strictly monotone) is classical. Suppose then $D(0)=\dot{D}(0)=0$. Fix $\varepsilon>0$ and define $\eta(\phi):=-\varepsilon D(\phi)$, $\phi \in (0, 1)$. We have
$$
\dot{\eta}(\phi) -h(\phi) + c + \frac{D(\phi)g(\phi)}{\eta(\phi)}\to -h(0) +c >0,\ \mbox{ as } \ \phi \to 0^+.
$$
Therefore $\eta$ is a strict upper-solution of $\eqref{first order problem000}_1$ in $(0,\delta]$, for some $\delta >0$.
Also, since $\dot{z}(0)=0$, there exists a sequence $\{\phi_n\}_n$, with $\delta \ge \phi_n \to 0^+$, such that $\dot{z}(\phi_n)\to 0$. From $\eqref{first order problem000}_1$, this implies that
$$
\lim_{n\to \infty} \frac{\varepsilon D(\phi_n)}{-z(\phi_n)}=\varepsilon\lim_{n\to \infty}\frac{\dot{z}(\phi_n) +c - h(\phi_n)}{g(\phi_n)}=\infty.
$$
Hence, $-\eta(\delta_1)=\varepsilon D(\delta_1) >-z(\delta_1)$, for some $0 < \delta_1 \le \delta$ small enough. An application of Lemma \ref{lem:cm-dpde} {\em (2.a.i)} then gives
\begin{equation}
\label{e:zD1}
z(\phi) > -\varepsilon D(\phi), \ \phi \in (0,\delta_1].
\end{equation}
This clearly implies that
$$
0> \frac{z(\phi)}{D(\phi)}> -\varepsilon, \ \phi \in (0,\delta_1].
$$
Since $\varepsilon>0$ is arbitrary, then we have $\phi\rq{}(\xi)\to 0$ for $\xi \to \xi_0^-$ and hence $\phi$ is classical, that is we showed the first part of {\em (ii)}. Define $\eta(\phi):=-\phi D(\phi)$. We have, for any $\phi \in (0,1)$,
\begin{equation*}
\dot{\eta}(\phi)-h(\phi) + c +\frac{D(\phi)g(\phi)}{\eta(\phi)}=-\dot{D}(\phi)\phi -D(\phi)-h(\phi)+c -\frac{g(\phi)}{\phi}.
\end{equation*}
Thus, by means of \eqref{e:g near zero}, we get
$$
\liminf_{\phi\to 0^+} \left[\dot{\eta}(\phi)-h(\phi) + c +\frac{D(\phi)g(\phi)}{\eta(\phi)}\right]=c-h(0)-\limsup_{\phi\to 0^+} \frac{g(\phi)}{\phi}>0.
$$
Therefore, $\eta$ is a strict upper-solution of $\eqref{first order problem000}_1$ in $(0,\delta]$, for some $\delta>0$. Furthermore, taking the same sequence $\phi_n\to 0^+$ as above such that $\dot{z}(\phi_n)\to 0$, as $n \to \infty$, then we have
$$
\liminf_{n\to \infty}
\frac{D(\phi_n)\phi_n}{-z(\phi_n)} =\liminf_{n\to \infty} \frac{\dot{z}(\phi_n)+c-h(\phi_n)}{g(\phi_n)/\phi_n}=\frac{c-h(0)}{\limsup_{n\to \infty}g(\phi_n)/\phi_n} >1,
$$
since \eqref{e:g near zero} holds. Thus, as in \eqref{e:zD1}, we deduce that $D(\phi)\phi > -z(\phi)$ in $(0,\delta]$, after choosing $0 < \delta \le 1/2$ small enough. Hence,
$$
\int_{0}^{1/2} \frac{D(\sigma)}{-z(\sigma)}\,d\sigma > \int_{0}^{\delta} \frac{d\sigma}{\sigma}=+\infty,
$$
which concludes the proof of {\em (ii)}, by means of \eqref{e:integral control}.

\smallskip
We show {\em (iii)}. By \eqref{derivative of z at zero}, \eqref{e:dotz 2}, $c^*>h(0)$ and \eqref{e:phisharp} we obtain $\dot{z}(0)=h(0)-c<0$. Then,
$$
\frac{D(\sigma)}{-z(\sigma)}=\frac{\dot{D}(0) + o(1)}{c-h(0) + o(1)}\ \mbox{ as } \ \sigma \to 0^+,
$$
and consequently \eqref{e:integral control} is verified. Thus, $\xi_0 \in \R$. Furthermore, from \eqref{e:z from phi},
$$
\lim_{\xi \to \xi_0^-} \phi\rq{}(\xi)=\lim_{\phi \to 0^+} \frac{z(\phi)/\phi}{D(\phi)/\phi}=\frac{h(0)-c}{\dot{D}(0)} \in [-\infty, 0),
$$
which implies that $\phi$ is sharp at $0$ and that \eqref{e:phiprime} holds.
\end{proofof}

\section{New regularity classification of wavefronts}
\label{s:rnew egularity of fronts}
\setcounter{equation}{0}

In this section we prove Theorem \ref{th:regularity wf}. Analogously to Section \ref{s:existence_0alpha}, but now thanks to assumptions (D0) - (g01), we apply results of Sections \ref{s:sing2} -- \ref{ssec:dotz at zero} to the case $q=Dg$.

\begin{proofof}{Theorem \ref{th:regularity wf}}
We first show that wavefronts are allowed if and only if $c \ge c^*$ for $c^*$ satisfying \eqref{e:est c*3}; the proof is mostly contained in the proof of Theorem \ref{th:swf to zero}. Then, we prove {\em (i)} and {\em (ii)}, by exploiting some of the arguments in the proof of Corollary \ref{cor:swf qualitative behavior}.

Set $q=Dg$. Clearly, $q$ satisfies (q), with in particular $\dot{q}(0)=0$. By Proposition \ref{prop: first order}, Problem \eqref{first order problem 0-0} admits a unique solution $z$ if and only if $c\ge c^*$ where for $c^*$ it holds \eqref{estimates on c*}. As observed in Remark \ref{rem:MP}, since (D0) and (g01) hold true, in this case $c^*$ satisfies \eqref{e:est c*3}.
\par
To the solution $z$ there is associated the solution  $\phi=\phi(\xi)$ of the problem
\begin{equation}
\label{cauchy3}
\begin{cases}
\phi\rq{}=\frac{z(\phi)}{D(\phi)},\\
\phi(0)=\frac1{2}.
\end{cases}
\end{equation}
Such a $\phi$ exists and satisfies $\eqref{cauchy3}_1$ in some maximal interval $(\xi_1,\xi_0)$, so that
\[
\lim_{\xi \to \xi_1^+}\phi(\xi)=1 \ \mbox{ and } \ \lim_{\xi \to \xi_0^-}\phi(\xi)= 0.
\]
 Also, $\phi$ satisfies \eqref{e:ODE} in $(\xi_1,\xi_0)$. As discussed in the proof of Theorem \ref{th:swf to zero}, if $\xi_0 \in \R$, then $\phi$ can be extended continuously to a solution of \eqref{e:ODE} in $(\xi_0,+\infty)$, by setting $\phi(\xi)=0$, for $\xi \ge \xi_0$. Since $g(1)=0$, it also holds that if $\xi_1\in \R$ then we can extend $\phi$ to a solution of \eqref{e:ODE} in $(-\infty, \xi_1)$, by setting $\phi(\xi)=1$ for $\xi\le \xi_1$. Thus, we can always consider $\phi$ satisfying weakly \eqref{e:ODE} in $\R$; moreover $\phi$ solves $\eqref{cauchy3}_1$ in $(\xi_1,\xi_0)$ with
\[
\xi_1=\inf\left\{ \xi \in \R: \phi(\xi)<1\right\}\in [-\infty, 0), \ \xi_0=\sup\left\{\xi \in \R: \phi(\xi)>0\right\} \in (0, + \infty],
\]
and it is constant in $\R\setminus (\xi_1,\xi_0)$. Thus, we showed that if $c\ge c^*$ then there exists a wavefront $\phi$ whose profile satisfies \eqref{e:infty}.

By reasoning as in the proof of Theorem \ref{th:swf to zero}, also the converse implication holds. Indeed, if $\phi$ is a profile of a wavefront satisfying \eqref{e:infty}, then the function $z$ defined by
$z(\phi):= D(\phi)\phi\rq{}\left(\phi^{-1}(\phi)\right)$, $0<\phi < 1$,
is a solution of \eqref{first order problem 0-0}. Thus, $c\ge c^*$.

\smallskip
We prove {\em (i)}. Assume $c>c^*$. From \eqref{derivative of z at zero} in Proposition \ref{prop:dotz at zero}, we have $\dot{z}(0)=0$. Hence, if $\dot{D}(0)\neq 0$ then it holds
\begin{equation}
\label{e:classical at zero}
\lim_{\xi \to \xi_0^-} \phi\rq{}(\xi)=\lim_{\phi \to 0^+} \frac{z(\phi)}{D(\phi)}=0.
\end{equation}
If $\dot{D}(0)=0$, then we argue as in the proof of Corollary \ref{cor:swf qualitative behavior}, see \eqref{e:zD1}, to show that, for any $\eps>0$ there exists $\delta\in (0,1)$ such that $z(\phi)> -\eps D(\phi)$, $\phi \in (0,\delta]$.
Hence,
\[
\lim_{\xi \to \xi_0^-}
\phi\rq{}(\xi) = \lim_{\phi \to 0^+}\frac{z(\phi)}{D(\phi)} \ge -\eps.
\]
Since $\phi\rq{}<0$ in $(\xi_1,\xi_0)$ and $\eps$ is arbitrarily small, it follows again \eqref{e:classical at zero}.
\par
We prove now {\em (ii)}. By $\eqref{derivative of z at zero}_2$, from $c=c^*>h(0)$ we have $\dot{z}(0)=h(0)-c^*<0$. Then,
$$
\frac{D(\sigma)}{-z(\sigma)}=\frac{\dot{D}(0) + o(1)}{c-h(0) + o(1)}\ \mbox{ as } \ \sigma \to 0^+,
$$
and consequently \eqref{e:integral control} is verified. Thus, $\xi_0 \in \R$. Furthermore, from \eqref{e:z from phi},
$$
\lim_{\xi \to \xi_0^-} \phi\rq{}(\xi)=\lim_{\phi \to 0^+} \frac{z(\phi)/\phi}{D(\phi)/\phi}=\frac{h(0)-c^*}{\dot{D}(0)} \in [-\infty, 0),
$$
and thus the conclusions hold.
\end{proofof}

\begin{remark}[{Case $c=c^*=h(0)$}]
\label{r:final?}
{
\rm
Part {\em (i)} and {\em (ii)} of Theorem \ref{th:regularity wf} do not cover the case $c=c^*=h(0)$. The following discussion shows that, to classify the behavior in that case, further assumptions are needed. More precisely, either a classical and a sharp wavefront can indeed occur under (D0) and (g01). Take $q$ and $h$ as in \eqref{e:qhrem} in Remark \ref{rem:c*=0}. There, we proved that in this case it holds $c^*=h(0)=0$. Consider
\begin{equation*}
\begin{cases}
D_1(\phi)= \phi^2, \\
g_1(\phi)=\phi(1-\phi),
\end{cases}
\
\begin{cases}
D_2(\phi)= \phi, \\
g_2(\phi)=\phi^2(1-\phi).
\end{cases}
\end{equation*}
Clearly, $D_1$ and $g_1$ satisfy (D0) and (g01) and so $D_2$ and $g_2$. Also, since $D_1g_1=q=D_2g_2$, then $c_1^{*}=c_2^{*}=h(0)=0$, where $c_1^{*}$ and $c_2^{*}$ are the thresholds given by Proposition \ref{prop: first order} associated with $D_1 g_1$ and $D_2g_2$, respectively.
Define, for $\xi \in \R$,
\begin{equation*}
\phi_1(\xi):=
\left\{
\begin{array}{ll}
1 - \frac{e^\xi}{2}, & \xi < \log(2),\\
0, & \mbox{ otherwise,}
\end{array}
\right.
\ \mbox{ and } \
\phi_2(\xi):= \frac{1}{1+e^{\xi}}.
\end{equation*}
Direct computations show that $\phi_1$ and $\phi_2$ are two wave profiles defining two wavefronts, both of them associated with $c=h(0)$. Plainly, $\phi_1$ is sharp at $\xi=\log(2)$ while $\phi_2$ is classical.
}
\end{remark}

\section*{Acknowledgments}
The authors are members of the {\em Gruppo Nazionale per l'Analisi Matematica, la Probabilit\`{a} e le loro Applicazioni} (GNAMPA) of the {\em Istituto Nazionale di Alta Matematica} (INdAM) and acknowledge financial support from this institution.

{\small
\bibliography{refe_BCM}
\bibliographystyle{abbrv2}
}

\end{document}